\documentclass{amsart}[14pt]

\usepackage{bm}

 \usepackage{upgreek}

\usepackage{stmaryrd} 
\usepackage[T1]{fontenc}
\usepackage{textcomp}

\usepackage{amsopn, amsthm, amsgen, amscd,amsmath,amssymb}
\usepackage[bbgreekl]{mathbbol}
\usepackage[mathcal]{euscript}
\DeclareMathAlphabet{\mathbbm}{U}{bbm}{m}{n}
\usepackage{color}
\usepackage{tikz}
  \usepackage{graphicx}
\usepackage{epstopdf}
\usepackage[small,nohug,heads = LaTeX]{diagrams} \diagramstyle[labelstyle=\scriptstyle]\definecolor{CadetBlue}{cmyk}{0.62, 0.57, 0.23, 0 }
\definecolor{RoyalBlue}{cmyk}{1, 0.5, 0, 0 }
\definecolor{RedViolet}{cmyk}{0.07, 0.9, 0, 0.34 }
\definecolor{SeaGreen}{cmyk}{0.69, 0, 0.5, 0}

\DeclareMathAlphabet{\mathpzc}{OT1}{pzc}{m}{it}

\newcommand{\A}{\mathbb A}
\newcommand{\R}{\mathbb R}
\newcommand{\C}{\mathbb C}

\newcommand{\K}{\mathbb K}
\newcommand{\N}{\mathbb N}

\newcommand{\Q}{\mathbb Q}
\newcommand{\Z}{\mathbb Z}

\newcommand{\T}{\mathbb T}

\newtheorem{theo}{Theorem}
\newtheorem{lemm}{Lemma}
\newtheorem{prop}{Proposition}
\newtheorem{coro}{Corollary}

\newtheorem*{theob}{Baker's Theorem}
\newtheorem*{lwtheo}{Lindemann-Weierstra\ss\ Theorem}
\newtheorem*{logconj}{Logarithm Conjecture}
\newtheorem*{schconj}{Schanuel Conjecture}

\theoremstyle{definition}

\theoremstyle{remark}

\newtheorem{exam}{Example}

\newtheorem{note}{Note}
\newtheorem{NAnote}{Nonarchimedean Note}

\newcommand{\bast}{{}^{\ast}}
\newcommand{\bbull}{{}^{\bullet}}

\title[]{Diophantine Approximation Groups, Kronecker Foliations and Independence}
\author{T.M. Gendron}
\address{Instituto de Matem\'{a}ticas -- Unidad Cuernavaca, Universidad
Nacional Aut\'{o}noma de M\'{e}xico, Av. Universidad S/N, C.P. 62210
Cuernavaca, Morelos, M\'{E}XICO}
\email{tim@matcuer.unam.mx}
\date{6 August 2012}
\begin{document}
\vspace{2cm}
\begin{abstract}  
We introduce
diophantine approximation groups and their associated Kronecker
 foliations, using them to provide new algebraic and geometric characterizations of $K$-linear and algebraic dependence.
As a consequence we find reformulations -- as algebraic and geometric (graph) rigidities -- of the Theorems of Baker and Lindemann-Weierstrass, the Logarithm Conjecture and the Schanuel Conjecture.  There is an Appendix describing diophantine approximation
groups as model theoretic types.
\end{abstract}
 \maketitle

\section*{Introduction}

This is the first of three papers introducing a paradigm within which global algebraic number theory
for $\R$ may be formulated, in such a way as to make possible the synthesis of
algebraic and transcendental number theory into a coherent whole.  
The picture offered here is evolving and not yet definitive in form, but is one consistent with
principles stated or implicit throughout number theory.

The proposed theory is based on the notion of a {\it diophantine approximation} of a real number $\uptheta$:
informally, an integer sequence $\{ n_{i}\}$ for which there is another integer sequence $\{ n^{\perp}_{i}\}$ satisfying
\[ n_{i}\uptheta -n^{\perp}_{i}\rightarrow 0.\]  
Actually we will
identify sequences that are ``eventually equal'' by passing to an ultrapower $\bast\Z$ of $\Z$ (c.f. \cite{BS}), defining
a diophantine approximation as a sequence class $\bast n\in\bast\Z$ satisfying:
\begin{align}\label{lessnaivedefDA} \upvarepsilon (\uptheta)(\bast n ):= \bast n\uptheta -\bast n^{\perp} \in\bast\R_{\upvarepsilon} ,  \end{align}
where $\bast n^{\perp}\in\bast\Z$, $\bast\R_{\upvarepsilon}\subset\bast\R$ is the subgroup of {\it infinitesimals} in a corresponding ultrapower
$\bast\R$ of the reals.

The set of $\bast n$ satisfying (\ref{lessnaivedefDA}) defines 
a subgroup
\[ \bast\Z (\uptheta ) < \bast\Z \] which may be thought of as the analogue notion of 
``principal ideal generated by the denominator of $\uptheta$'',
and the association $ \bast n\mapsto \upvarepsilon (\uptheta)(\bast n)$ defines a homomorphism, the {\it error map}. 
The nonstandard model $\bast\Z$ represents the ``diophantine approximation of the ring of $\R$-integers''.
See \S \ref{ratapprox}.

The group $\bast\Z (\uptheta )$ has an important geometric interpretation: it is the {\it fundamental germ} (a foliated notion of fundamental group, \cite{Ge1}, \cite{Ge2}) 
of the Kronecker foliation $\mathfrak{F}(\uptheta )$.  Specifically,
the real vector space \[ \bbull\R :=\bast\R/
\bast\R_{\upvarepsilon} \]plays the part of foliation universal cover  and \[  \mathfrak{F}(\uptheta )\cong \bbull\R /\bast\Z (\uptheta ) .\]
The group $\bast\Z (\uptheta )$ was used in \cite{CaGe1} to define a notion of modular invariant for quantum tori; in this article,
we view $\mathfrak{F}(\uptheta )$ as the analogue of  ``residue class field'' of $\bast\Z (\uptheta )$.
This is discussed in \S \ref{ResClassGps}.

The question of determining what class of number $\uptheta$ is in terms of properties of  $\bast\Z (\uptheta )$
is one which will be pursued systematically throughout the first two papers in this series.
In particular, in \cite{Ge5}, we will characterize in terms of {\it ideological arithmetic} the sets of
algebraic numbers, badly approximable numbers, well approximable numbers, Liouville numbers and the Mahler classes $S$, $T$, $U$.
In this paper, our purpose is three-fold:
\begin{itemize}
\item Develop the basic theory of diophantine approximation groups of real numbers by rational integers, \S\S  \ref{ratapprox},\ref{ResClassGps},  extending it to:
\begin{itemize}
\item Diophantine approximations groups of real matrices, \S \ref{higher}.
\item Diophantine approximation of real numbers and real matrices in rings of integers, \S \ref {daringint}
\item Diophantine approximation rings (of real numbers, real vectors) in rings of polynomials, \S \ref{polydio}.
\end{itemize}
\item Provide
novel algebraic and geometric readings of $K$-linear and algebraic dependence: 
a dictionary translating: $K$-linear and algebraic independence $\leftrightarrows$ algebraic properties of diophantine
approximations groups $\leftrightarrows$ geometric properties of Kronecker foliations, \S\S \ref{higher}--\ref{polydio}.  
\item Use the dictionary to reformulate classical theorems/conjectures of transcendental number theory
as (graph) rigidities, expressed in terms of diophantine
approximations groups and  Kronecker foliations, \S \ref{rigidity}.
\end{itemize}

The first entry of the dictionary is:

\vspace{3mm}
\noindent\fbox{{\small {\bf Irrationality}}}
\begin{center}
    \begin{tabular}{ | p{2.8cm} | p{2.8cm} | p{3cm} | p{2.8cm}|}
    \hline
    $\quad\quad\quad\quad\uptheta$ & $\quad\quad\quad\;\;\upvarepsilon (\uptheta )$ & $\;\;\quad\quad\;\;\bast\Z (\uptheta )$ & $\quad\quad\quad\;\;\mathfrak{F}(\uptheta )$ \\ \hline
   {\small irrational} & {\small injective}  & {\small contains no $\bast\Z$ ideales} & {\small 1-connected leaves}   \\ \hline
    \end{tabular}
\end{center}

\vspace{3mm}
 

Further entries are obtained by considering diophantine approximation groups of real matrices,
or by extending the approximants to algebraic integers or integer polynomials.  

For a matrix 
$\Uptheta\in M_{r,s}(\R)$ we define the diophantine approximation group \[ \bast\Z^{s}(\Uptheta )\subset \bast\Z^{s}\]
by the condition \[ \Uptheta\bast\boldsymbol n -\bast{\boldsymbol n}^{\perp}\in \bast\R_{\upvarepsilon}^{r},\]
as well as the subgroup \[ \bast\widetilde{\Z}^{s}(\Uptheta )<\bast\Z^{s}(\Uptheta )\]of $\bast{\boldsymbol n}$ for which
$\bast{\boldsymbol n}^{\perp} =0$.  To these groups are associated error maps $\upvarepsilon (\Uptheta)$ and 
$\tilde{\upvarepsilon} (\Uptheta)$ and
a pair of Kronecker foliations, \[ \mathfrak{F}(\Uptheta ),\quad \mathfrak{f}(\Uptheta )\] which foliate $\T^{r+s}$ resp.\ $\T^{s}$.  With these objects $\Q$-linear independence translates as:

\vspace{3mm}

\noindent\fbox{{\small {\bf $\Q$-Independence}}}
 \begin{center}
    \begin{tabular}{ | p{4cm} | p{2.5cm} | p{2.5cm} | p{2.5cm}|}
    \hline
    $\quad\quad\quad\quad\quad\Uptheta$ &  {\small error map}  & {\small DA group} & {\small Foliation} \\ \hline
    {\small inhomogeneously 
     $\Q$ independent columns}   & {\small $ \upvarepsilon (\Uptheta )$ is  injective}  & {\small $\bast\Z^{s}(\Uptheta )$ contains no $\bast\Z$ ideals} & {\small leaves of
     $\mathfrak{F}(\Uptheta )$ are 1-connected}  \\ \hline
     {\small homogeneously 
     $\Q$ independent columns } & {\small $\tilde{\upvarepsilon} (\Uptheta)$ is  injective}  & {\small $\bast\widetilde{\Z}^{s}(\Uptheta )$ contains no $\bast\Z$ ideals} & {\small leaves of
     $\mathfrak{f}(\Uptheta )$ are 1-connected } \\ \hline
    {\small homogeneously 
     $\Q$ independent rows} &   &  &
     {\small $\mathfrak{F}(\Uptheta )$ has dense leaves} \\ \hline
    \end{tabular}
\end{center}

\vspace{3mm}

The connection with classical algebraic number theory is made by means of diophantine approximation
with respect to the ring $\mathcal{O}$ of $K$-integers of a finite extension $K/\Q$.  
 If $\K$ denotes the Minkowski space 
(the archimedean factor of the $K$-adeles), then $\mathcal{O}\subset\K$ is a lattice, and using the diagonal embedding $\R\hookrightarrow \K$ we may define diophantine approximation groups\footnote{The reader should be advised that our notion of approximation differs in an essential way from the 
one which is encountered in the classical Diophantine Approximation by algebraic integers \cite{Bu}, \cite{Sch}, where one considers approximations of $\uptheta\in\R$ by sequences of elements of $K$ with respect to a fixed real place.  The approach adopted here is lattice theoretic:  $\uptheta$, diagonally embedded in $\K$, is approximated by elements of $\mathcal{O}$ only, embedded in $\K$ by {\it all} of the archimedean places.} \[ \bast\mathcal{O}(\uptheta )\supset \bast \Z (\uptheta ),\] where the error
 $\upvarepsilon (\uptheta)(\bast\upalpha )$ belongs to $\bast\K_{\upvarepsilon}$.  
If $K/\Q$ is Galois, the Galois group ${\rm Gal}(K/\Q )$ acts by automorphisms on $\bast\mathcal{O}(\uptheta )$ fixing
$\bast \Z (\uptheta )$: this representation of the Galois theory of algebraic number fields is where the fusion between algebraic and transcendental number theory takes place; it
will be explored systematically in \cite{Ge6}.  

The groups $\bast\mathcal{O}(\uptheta )$ have a
kinship with Diophantine Approximation by conjugates \cite{RoWa}.  
To make this connection concrete it is necessary to pass to the isomorphic group of denominator numerator pairs 
\[ \bast\mathcal{O}^{1,1}(\uptheta )=\{ (\bast\upalpha, \bast \upalpha^{\perp} )|\; \bast\upalpha\in \bast\mathcal{O}(\uptheta )\}\subset\bast\mathcal{O}^{2}.\]  For each archimedean place $\upnu$
we define $\bast \mathcal{O}^{1,1}(\uptheta )_{\upnu}$ as the group of pairs defining a diophantine approximation
in the coordinate $\upnu$ only: then we have the factorization
\[ \bast\mathcal{O}^{1,1}(\uptheta )=\bigcap \bast \mathcal{O}^{1,1}(\uptheta )_{\upnu},\]
so that the elements of $\bast\mathcal{O}^{1,1}(\uptheta )$ consist of approximations by all conjugates at once.
The  $\upnu$-local versions have corresponding Kronecker 
foliations, and consideration of these objects yield new dictionary entries for $K$-independence generalizing those displayed above, \S \ref {daringint}.

Questions concerning transcendentality and more generally of algebraic independence 
are coded by the ring of polynomial diophantine approximations.  If $\boldsymbol \uptheta = (\uptheta_{1},\dots ,\uptheta_{s})$
is a set of real numbers then \[\bast\Z[\boldsymbol X](\boldsymbol \uptheta)\subset \bast\Z[\boldsymbol X]\] is defined as the set of polynomials $F(\boldsymbol X)$
satisfying
\[ F(\boldsymbol \uptheta )\in\bast\R_{\upvarepsilon}.\]  There is an error map $\upvarepsilon^{\rm poly}(\boldsymbol \uptheta)$
and a Kronecker foliation
$\mathfrak{F}^{\rm poly}(\boldsymbol \uptheta)$ having infinite dimensional leaves, and we obtain another
entry regarding algebraic independence:

\vspace{3mm}

\noindent\fbox{{\small {\bf Algebraic Independence}}}
\begin{center}
    \begin{tabular}{ | p{3.5cm} | p{2.7cm} | p{2.7cm} | p{2.7cm} |}
    \hline
    $\quad\quad\quad\quad\quad\boldsymbol \uptheta$ & {\small $\quad\quad\;\;\upvarepsilon^{\rm poly}(\boldsymbol \uptheta)$} & {\small $\quad\quad\bast\Z[\boldsymbol X](\boldsymbol \uptheta)$} & 
    {\small $\quad\quad\mathfrak{F}^{\rm poly}(\boldsymbol \uptheta)$} \\ \hline
    {\small algebraically independent coordinates} & {\small injective}  & {\small contains no $\bast\Z[\boldsymbol X]$ ideals } & {\small 1-connected leaves}  \\ \hline
    \end{tabular}
\end{center}

\vspace{3mm}

\vspace{5mm}




In \S \ref{rigidity}, new rigidity formulations of the theorems of Lindemann-Weierstra\ss\ and Baker as well as the Logarithm and Schanuel conjectures are obtained by running their statements through the rows of the above dictionary entries.  For example, letting 
\[ \mathcal{L} =\{ z\in \C|\; \exp (z)\in\overline{\Q}\}\]  be the $\Q$-vector space of logarithms of algebraic numbers, the real version of the Logarithm conjecture translates as:

\vspace{3mm}
\noindent\fbox{{\small {\bf $\R$ Logarithm Conjecture}}}   For all $\boldsymbol \uptheta\in \mathcal{L}^{n}\cap\R^{n}$
\begin{center}
    \begin{tabular}{ | p{3cm} | p{2.7cm} | p{2.7cm} | p{3cm} |}
    \hline
    {\small Classical formulation} & {\small error maps} & {\small DA groups} & 
    {\small Kronecker foliations} \\ \hline
    {\small algebraically dependent coordinates $\Rightarrow$ $\Q$-dependent coordinates} & {\small $  {\rm Ker}\left( \upvarepsilon^{\rm poly}(\boldsymbol \uptheta ) \right)\not=0$ $\Rightarrow$
 ${\rm Ker}\left( \tilde{\upvarepsilon} (\boldsymbol \uptheta ) \right)\not=0$}  & {\small $\bast\Z[\boldsymbol X](\boldsymbol\uptheta ]$ contains a non-0 $\bast\Z[\boldsymbol X]$-ideal
$\Rightarrow$ $\bast\widetilde{\Z}^{n}(\boldsymbol\uptheta )$ contains a non-0 $\bast\Z^{n}$-ideal}  &  {\small $\mathfrak{F}^{\rm poly}(\boldsymbol \uptheta )$ has non 1-connected
leaves $\Rightarrow$ $\mathfrak{f}(\boldsymbol \uptheta)$
has non 1-connected
leaves}  \\ \hline
    \end{tabular}
\end{center}

\vspace{3mm}




\vspace{3mm}

The Baker and Lindemann-Weierstra\ss\ Theorems have comparable rigidity translations, while the translation of
the Schanuel Conjecture will be
made as a {\it graph rigidity} with respect to the exponential function: the latter implies a family of rigidities indexed by the coordinate projections
of the graph of the exponential.  

As Kronecker foliations are essentially abelian varieties with the additional data of a fixed continuous subgroup,
it is our belief that the geometric columns of these rigidity reformulations will make it possible to view the aforementioned class of statements
in the spirit of Raynaud's Theorem (the Manin-Mumford conjecture) and more generally, the Mordell-Lang conjectures.

The paper ends with a short Appendix in which diophantine approximation groups are interpreted as types:
model theoretic generalizations of ideals.  While this observation is not used in the body proper of the paper,
it nevertheless serves to reinforce our general outlook.

\vspace{5mm} {\bf Acknowledgements:}  I would like to thank J. Baldwin for help in the writing of the Appendix. This paper was supported in part by CONACyT grant 058537
as well as the PAPIIT grant
IN103708.

\section{Nonstandard Structures}\label{nonstd}

Fix a  non-principal ultrafilter $\mathfrak{u}$ on $\N$ \cite{Co}: a maximal subset of $2^{\N}$ that contains only infinite subsets and is closed with respect to intersection and 
upward inclusion (if $X\in\mathfrak{u}$ and
$X\subset Y$ then $Y\in\mathfrak{u}$).  
 Let $G$ be an algebraic structure such as a group, ring or field.
Then we may define an equivalence relation on the set of sequences $G^{\N}$ by
\[ \{g_{i}\}\sim_{\mathfrak{u}} \{h_{i}\} \Leftrightarrow \{ i|\; g_{i} = h_{i}\}\in\mathfrak{u}.\]
The
 ultrapower \cite{BS} is the quotient
\[   \bast G =\bast G_{\mathfrak{u}}:= G^{\N}/\sim_{\mathfrak{u}}.\]

The operations of $G$ extend component-wise to corresponding operations on $\bast G$, and there is 
a canonical monomorphism $G\hookrightarrow\bast G$ given by constant sequences.
If $G$ is a group, ring or field, the same is true of $\bast G$ (this follows from \L o\'{s}'s Theorem \cite{Ho}).
Ultrapowers associated to distinct nonprincipal ultrafilters need not be isomorphic; generally, the problem is one of cardinality, which may be overcome if set-theoretic assumptions such as the GCH are imposed.  This will not concern us in this paper.

We will mainly be interested in two cases: $G=\Z$ or $\R$.  We will often refer to $\bast\Z$, $\bast\R$ as the {\bf {\small nonstandard
integers}} resp.\ the {\bf {\small nonstandard reals}}.  We remark that $\bast\Z$ is not a principal ideal domain \cite{Ch},
and that the cardinality of $\bast\Z$ is that of the continuum
\cite{Ke}.
The field $\bast\R$  
is non-archimedean and contains a local subring 
\[ \bast\R_{\rm fin}  =\big\{ \text{classes of bounded sequences}\big\}\]
whose maximal ideal consists of the infinitesimals
\[  \bast\R_{\upvarepsilon} =\big\{ \text{classes of sequences converging to }0\big\}.\] 
We write 
\[ \bast r\simeq\bast s\] if $\bast r-\bast s\in \bast\R_{\upvarepsilon}$ and
say then that the two nonstandard reals are {\bf {\small asymptotic}} or {\bf {\small infinitesimal}}.  

For any $\bast r \in \bast\R_{\rm fin} $  there is a
unique limit point $r$ which is the limit of a subsequence
 $\{ r_{i}\}|_{i\in X}$ with $X\in\mathfrak{u}$; this $r$ is independent of $\{ r_{i}\}\in \bast r $ and is 
 called the {\bf {\small standard part}} of $\bast r$, denoted
\[ {\sf std}(\bast r).\]

\begin{prop} The standard part map
\begin{align}\label{standardpartepi} {\sf std}: \bast\R_{\rm fin} \longrightarrow\R\end{align}
is a well-defined epimorphism
of rings with kernel $\bast\R_{\upvarepsilon}$.  In particular, it induces an isomorphism between
the residue class field
 $\bast\R_{\rm fin}/\bast\R_{\upvarepsilon}$ and $\R$.
 \end{prop}
\begin{proof} That ${\sf std}$ is an epimorphism is trivial, as is the inclusion
$\bast\R_{\upvarepsilon}\subset {\rm Ker}({\sf std})$.  If $\{r_{i}\}\in \bast r \in {\rm Ker}({\sf std})$ then there exists $X\in\mathfrak{u}$ so that $\{ r_{i}\}|_{i\in X}\rightarrow 0$.  If we complete $\{ r_{i}\}|_{i\in X}$ to an 
$\N$-indexed sequence $\{ r'_{i}\}$ converging to $0$, then $\{ r'_{i}\}\sim_{\mathfrak{u}} \{r_{i}\}$ hence $\bast r\in \bast\R_{\upvarepsilon}$. 
\end{proof}
Thus we see the dual nature of $\R$, which occurs
as both completion and residue class field.   
See \cite{Go}, \cite{HL}, \cite{Ro} for a general discussion of the field of nonstandard reals.


\begin{NAnote}  The restriction of ${\sf std}$ to $\bast\Q_{\rm fin}:=\bast\Q\cap\bast\R_{\rm fin}$ also maps onto $\R$,
with kernel $\bast\Q_{\upvarepsilon}:=\bast\Q\cap \bast\R_{\upvarepsilon}$.  In fact, a similar structure exists in
$\bast\Q$ for the $p$-adic completion
$\hat{\Q}_{p}$ of $\Q$.  Using the $p$-adic absolute value to define finiteness and infinitesimality, we can associate to each $p$
a local ring $\bast \Q_{p,{\rm fin}}\subset\bast\Q$ with maximal ideal $ \bast \Q_{p,\upvarepsilon}$ and
residue class field $\hat{\Q}_{p}$.
We have 
\[ \bast \Q_{p,\upvarepsilon}\subset\bast\Z\subset \bast \Q_{p,{\rm fin}},\]
and the standard part map ${\sf std}_{p}:\bast \Q_{p,{\rm fin}}\rightarrow \hat{\Q}_{p}$ restricted to
$\bast\Z$ gives an epimorphism $\bast\Z\rightarrow \hat{\Z}_{p}$ = the ring of $p$-adic integers (onto since $\Z$ is dense in $\hat{\Z}_{p}$).  Thus, in the nonarchimedean case, where
there is already a satisfactory notion of ring of integers, nothing is gained by passing to $\bast\Q$.  By contrast, 
in the archimedean case, access to $\bast\Z$ -- 
which in the absence of a classical notion of $\R$-integers will {\it ipso facto} assume the role  -- is obtained only by 
transcending $\bast\Q_{\rm fin}$ and working within the realm of the archimedean infinite. 
\end{NAnote}

Define the  {\bf {\small extended reals}} to be the quotient
of $\R$-vector spaces
\[  \bbull\R :=\bast\R/\bast\R_{\upvarepsilon}.\] 
The product in $\bast\R$ does not preserve the relation
 $\simeq$ and so $\bbull\R$ inherits only
a real vector space structure, which, for the time being, is bereft of any particular topology.
The inclusions $\bast\Z,\R\hookrightarrow
\bast\R$ survive the quotient by $\bast\R_{\upvarepsilon}$ and yield inclusions
$\bast\Z, \R\hookrightarrow \bbull\R$: a subring and subfield respectively, whose intersection is $\Z$.
In particular, $\bbull\R = \R+\bast\Z$.

In \S \ref{ResClassGps}, we will regard $\bbull\R$ as a foliated space whose leaves are the cosets of $\R$, equipped with a distinguished complete transversal $\bast\Z$
(a transversal since its intersection with any leaf is countable and discrete in the leaf topology).
This foliated structure will be summarized by the triple 
\[ \hat{\R} := (\bbull\R, \R, \bast\Z ),\] 
the real numbers ``retrofitted
with integers from without''.  

\begin{NAnote}
The image of $\bast\Z$ in $\bbull\hat{\Q}_{p}:=\bast\hat{\Q}_{p}/\bast\hat{\Q}_{p,\upvarepsilon}$ is $\hat{\Z}_{p}\subset\hat{\Q}_{p}$, which does
{\it not} define a transversal for a foliated space structure on $\bbull\hat{\Q}_{p}$.
\end{NAnote}

  \section{Diophantine Approximation Groups of Real Numbers}\label{ratapprox}

Let $\uptheta\in \R$.    The following definition was given in  \cite{Ge1}, \cite{Ge2}.
A {\bf {\small  diophantine approximation}}
of $\uptheta$ is an element $\bast n\in\bast\Z$ for which there exists
$\bast n^{\perp}=\bast n^{\perp_{\uptheta}} \in\bast\Z$ such that
\begin{equation}\label{dioap} \uptheta\bast n\simeq\bast  n^{\perp} .  
\end{equation}
  The {\bf {\small  group of diophantine approximations}}
of $\uptheta$ is 
\[ \bast\Z (\uptheta ):= \big\{ \bast n\in\bast\Z\text{ is a diophantine approximation of }\uptheta \big\}  \]
clearly a subgroup of $\bast\Z$ since the defining infinitesimal equations
(\ref{dioap}) may be added.  

 If we pass to the quotient vector space $\bbull\R$ then the infinitesimal equation
(\ref{dioap}) becomes
 \[ \uptheta\bast n = \bast n^{\perp} \in\bast\Z,\] and
we may identify $\bast\Z (\uptheta )$ in $\bbull\R$ 
as the group
\[  \bast\Z (\uptheta ) = \big( \uptheta\cdot \bast\Z\big) \cap \bast\Z .\] 
Thus in $\bbull\R$
\begin{equation}\label{fracrep}  \uptheta = \bbull\text{-class of }\frac{\bast n^{\perp}}{\bast n} .
\end{equation}
In fact, every $\bbull r\in\bbull\R$ may be represented as a $ \bbull$-class of a quotient of elements
of $\bast\Z$; however, unlike the infinitesimal equations (\ref{dioap}), we cannot manipulate arithmetically the representations (\ref{fracrep}).

\begin{exam}\label{bestapprox} If $\{(p_{i}, q_{i})\}$ is the sequence of best approximations of $\uptheta\in\R-\Q$ (see 
 \cite{Ca}, \cite{Sch}), then the sequence of denominators $\{ q_{i}\}$ defines a class
$  \bast q_{\uptheta} \in \bast\Z (\uptheta ) $. 
e.g.\ if $\uptheta = (1+\sqrt{5})/2$ (the golden mean), then $ \bast q_{\uptheta}$ is the class defined
by the Fibonacci sequence.  
\end{exam}

Consider 
\[  \bast\Z (\uptheta )^{\perp} = \{ \bast n^{\perp}|\; \bast n\in \bast\Z (\uptheta )\},\]
the {\bf {\small group of dual diophantine approximations}}.
Then if $\uptheta\not=0$,
\[ \bast\Z (\uptheta )^{\perp}=\bast\Z(\uptheta^{-1}),\] which may be seen by
multiplying both sides of (\ref{dioap}) by $\uptheta^{-1}$. 
Moreover, the duality map $\bast n\mapsto \bast n^{\perp}$
defines an isomorphism 
\[  \perp=\perp_{\uptheta}: \bast\Z (\uptheta )\longrightarrow \bast\Z (\uptheta )^{\perp}\] for $\uptheta\not=0$.

At times (viz.\ \S \ref{daringint}), it will be useful to remember the dual of each element $\bast n\in\bast\Z (\uptheta)$ and 
 consider instead
the group of denominator numerator pairs:
\[ \bast\Z^{1,1}(\uptheta) = \big\{ (\bast n , \bast n^{\perp})\bigg|\; \bast n\in\bast\Z (\uptheta ) \big\}
\subset\bast\Z^{2}\] 
canonically isomorphic to $\bast\Z (\uptheta )$ via $ 
(\bast n , \bast n^{\perp})\mapsto\bast n$.  Note that if $\uptheta\not=\upeta$ then 
\[\bast\Z^{1,1}(\uptheta)\cap \bast\Z^{1,1}(\upeta) = (0,0).\]

\begin{theo}\label{groupideal} The subgroup $\bast\Z (\uptheta )$ is an ideal of $\bast\Z$
if and only if $\uptheta \in\Q$, and in this case, if $\uptheta = a/b$ (written in lowest terms) then
\[\bast\Z (\uptheta ) =\bast (b)\cong\bast\Z.\]
\end{theo}

\begin{proof}  See also \cite{Ge1}.  By Kronecker's Theorem, if
$\uptheta\in\R-\Q$ then for any integer $N$, the image in $\T=\R/\Z$ of the set of multiples $N\Z\uptheta$ is dense.  Now $\bast n\in\bast\Z (\uptheta )$ $\Leftrightarrow$
$\bast n$ is represented by $\{n_{i}\}$ for which the image of $\{n_{i}\uptheta\} $ in $\T$ converges to $0$.  For each
$i$ choose $N_{i}$ so that the image of $N_{i}n_{i}\uptheta$ in $\T$ is of distance greater than $1/4$ from $0$.
Then the class $\bast N\bast n$ of $\{ N_{i}n_{i}\}$ does not belong to $\bast\Z (\uptheta )$.  On the other hand, if
$\uptheta =a/b\in\Q$, written in lowest terms, then it is easy to see that $\bast n\in \bast\Z (\uptheta )$ $\Leftrightarrow$
$\bast n\in \bast (b)$, the latter being an ideal in $\bast\Z$.
\end{proof} 

Thus, integral ideals of $\Z$ are recovered as ultrapowers of dual diophantine approximation groups of rational numbers.  In particular if $\uptheta= b\in\Z$ then
$  \bast \Z(b)^{\perp} = \bast (b)$. 

\begin{prop}\label{cardcont}  For all $\uptheta$, $\bast\Z (\uptheta )$ has the cardinality of the continuum.
\end{prop}

\begin{proof}  
Since $\bast\Z$ has the cardinality of the continuum it will be enough to define an identification of $X=(0,1)-\Q$ with a subset of $\bast\Z (\uptheta )$.  For $x=0.d_{1}d_{2}\dots \in X$ define
\[ D_{i}=d_{1}10+\cdots +d_{i}10^{i}.\]  We may find for each $i$, $n_{i}(x)\in\Z$ with $10^{D_{i}}<n_{i}(x)<10^{D_{i}+1}$
such that $\{ n_{i}(x)\}$ defines an element $\bast n(x)\in\bast\Z (\uptheta )$. The map $x\mapsto\bast n(x)$ is injective: if
$y=0.d_{1}'d_{2}'\dots $ and
$i_{0}$ is the first index with $d_{i_{0}}\not=d'_{i_{0}}$ then
$ n_{i}(x)\not= n_{i}(y)$ for all $i\geq i_{0}$.  Since $\mathfrak{u}$ contains no finite sets, 
$X=\{ i\geq i_{0}\}\in\mathfrak{u}$ i.e.\ $\bast n(x)\not=\bast n(y)$.
\end{proof}


\begin{NAnote}\label{padic} The analogue of
diophantine approximation groups in the $p$-adic setting simply reproduces the ideal theory of $\hat{\Z}_{p}$. 
Indeed, for $\upxi\in\Q_{p}$, let 
\[ \bast\Z(\upxi )=\{ \bast n\in \bast\Z| \exists
\bast m\in \bast\Z \text{ such that } \bast n\upxi-\bast m\in\bast\hat{\Q}_{p,\upvarepsilon}\}.\]  Since $\Z\subset\hat{\Z}_{p}$ is dense,
it follows that if $\upxi\in\hat{\Z}_{p}$ then $\bast\Z(\upxi )=\bast\Z$.  On the other hand, if $|\upxi |_{p}=p^{m}>1$
then 
$ \bast\Z(\upxi )=\bast (p^{m}) $.  Thus every diophantine approximation group is the ultrapower of a $p$ power ideal in $\Z$: 
a pre-image of an ideal of $\hat{\Z}_{p}$ with respect to the epimorphism ${\sf std}_{p}:\bast\Z\rightarrow \hat{\Z}_{p}$.
\end{NAnote}

Let $\bast n\in\bast\Z (\uptheta )$. The associated {\bf {\small error}} is
\[ \upvarepsilon(\uptheta)(\bast n ):=
\uptheta\bast n - \bast n^{\perp},\]  and the
{\bf {\small error group}}
is 
\[   \bast\R_{\upvarepsilon}(\uptheta ) = \{    \upvarepsilon(\uptheta)(\bast n )|\; \bast n\in  \bast\Z (\uptheta )  \} \subset \bast\R_{\upvarepsilon}.\] 
The following Proposition gives the second column in the irrationality dictionary in the Introduction.
\begin{prop}\label{Qerrortermprop} When $\uptheta\in\Q$, $ \bast\R_{\upvarepsilon}(\uptheta )=0$, otherwise the homomorphism 
\[  \upvarepsilon (\uptheta):\bast\Z
(\uptheta )\rightarrow
 \bast\R_{\upvarepsilon}(\uptheta ),\quad \bast n\mapsto \upvarepsilon (\uptheta) (\bast n)\]  is an isomorphism.
 \end{prop}
 
 \begin{proof} The only diophantine approximations of $\uptheta\in\Q$ are those for which $\uptheta = \bast n^{\perp}/\bast n$
 (in $\bast \R$), this because the image of $\Z\uptheta$ in $\T$ is discrete.  On the other hand, if $\uptheta\in\R-\Q$
 then $\upvarepsilon (\uptheta)(\bast m)=\upvarepsilon(\uptheta)(\bast n)$ implies that $\bast m-\bast n\in\bast\Z(\uptheta)$
 satisfies $(\bast m-\bast n)\uptheta \in\bast\Z$, impossible by the irrationality of $\uptheta$.
  \end{proof}
  
Recall that $\uptheta,\upeta\in\R$ are {\bf {\small equivalent}} if there exists $A\in {\sf PGL}(\Z)$ with $A(\uptheta )=\upeta$.

\begin{theo}\label{PGLact} Let $A\in {\sf PGL}(\Z)$ be the class of the matrix
\[   A=\left(\begin{array}{cc}
a & b \\
c & d
\end{array}\right).\]
Then $A$ induces isomorphisms $A:\bast \Z(\uptheta )\longrightarrow \bast\Z (\upeta)$, 
$A^{\perp}:\bast \Z(\uptheta )^{\perp}\longrightarrow \bast\Z (\upeta)^{\perp}$ defined respectively by
$\bast n\mapsto c\bast n^{\perp}+d\bast n$, $\bast n^{\perp}\mapsto a\bast n^{\perp}+b\bast n$.
\end{theo}

\begin{proof} Follows from the calculation:
\begin{align*}
\upeta\cdot A(\bast n)-A^{\perp}(\bast n^{\perp}) & = 
\frac{1}{c\uptheta+d}\bigg\{(a\uptheta+b)\big(c\bast n^{\perp}+d\bast n\big)-(c\uptheta+d)\big(a\bast n^{\perp}+b\bast n \big)  \bigg\} \\
& = \frac{1}{c\uptheta+d} \big(\uptheta\bast n-\bast n^{\perp}\big) \\
& = \frac{\upvarepsilon (\bast n)}{c\uptheta+d} .
\end{align*}

\end{proof}

\begin{note}  The isomorphisms of Theorem \ref{PGLact} are obtained from projecting
the matrix action of $A\in {\rm GL}(2,\Z )$ on $\bast\Z^{1,1}(\uptheta )\subset\bast\Z^{2}$ onto the second resp. first coordinates.
\end{note}
 

One may also exchange the roles of $\R$ and $\bast\Z$ and consider for fixed $\bast m\in \bast\Z$
 the set of $\uptheta\in \R$ for which $\bast m$ defines a diophantine approximation:
\[ \R (\bast m )= \big\{ \uptheta\in\R |\; \bast m \in \bast\Z (\uptheta )  \big\}\supset\Z. \]
Note that $\R (\bast m )$ is a subgroup of $(\R,+)$,  which we call
the {\bf {\small group of real approximations}}
of $\bast m$.   When $\bast m =m\in\Z$
we have that $\R (m ) = \Z [1/m]$, the fractional ideal generated by $1/m$.
Otherwise, if $\Z(\bast m)$ contains some $\uptheta\in\R-\Q$, then $\langle 1,\uptheta \rangle\subset \R(\bast m)$
and therefore $ \Z(\bast m)$  is a dense subgroup of $\R$.  In particular, $\T(\bast m):=\R/\R (\bast m)$ is a quotient of the quantum torus
$\T(\uptheta)=\R/\langle 1,\uptheta \rangle$.  The group 
$\R (\bast m )$ provides another sense in which we may regard $\bast m\in\bast\Z$ as a ``real integer".
In \cite{Ge5}, \cite{Ge6} we will see that these groups are transverse to the $\bast\Z (\uptheta )$, and are equipped
with certain transverse metrics whose consideration is closely tied to the Littlewood conjecture.

The remainder of this section is devoted to the study of {\bf {\small extremal diophantine approximation groups}}.  
In what follows, if $n\Z\subset\Z$ is an ideal, we write $\bast (n\Z)(\uptheta )
=\bast (n\Z) \cap\bast\Z (\uptheta )$.  

For 
any $\uptheta\in\R$ and $N\in \Z$ we have
\begin{align}\label{multiplemapincl} \bast \Z(N\uptheta)\supset \bast\Z (\uptheta )\quad\text{and}\quad\bast \Z(N\uptheta)^{\perp}\subset \bast\Z (\uptheta )^{\perp}.\end{align}
as well as
\begin{align}\label{multiplemapincl2} \bast \Z(\uptheta+N)=\bast\Z (\uptheta )\quad\text{and}\quad\bast \Z(\uptheta+N)^{\perp}\subset \bast\Z (\uptheta )^{\perp}+N \bast \Z(\uptheta).\end{align}
In particular, if we denote by \[ {\sf Aff}(\Q) = \{a(x)=qx+q'|\; q\in\Q^{\ast},q'\in\Q\}\]the affine group over $\Q$, then the collection
\[ \{ \bast \Z(a(\uptheta))\}_{a(x)\in {\sf Aff}(\Q)}\] is directed both inversely and directly with respect to inclusion and we can form
the {\bf {\small initial}} and {\bf {\small final diophantine approximation groups}} as the limits:
\[  \bast\Z(\hat{\uptheta}) := \bigcap_{a(x)\in {\sf Aff}(\Q) } \bast \Z(a(\uptheta)), \quad 
\bast\Z(\check{\uptheta}) :=\bigcup_{a(x)\in {\sf Aff}(\Q) } \bast \Z(a(\uptheta)) .\]
In fact, the subset indexed by the homotheties $\Q^{\ast}\subset {\sf Aff}(\Q)$ (maps of the form $a(x)=qx$) produces a  cofinal family and the above limits 
reduce accordingly e.g. \[ \bast\Z(\hat{\uptheta}) = \bigcap_{q\in\Q^{\ast} } \bast \Z(q\uptheta).\]


By the elementary existence result of simultaneous Diophantine Approximation
(e.g.\ Theorem VI on page 13 of \cite{Ca}), $ \bast\Z(\hat{\uptheta}) $ is non-trivial.
And 
$\bast\Z(\check{\uptheta})$ is a proper subgroup of $\bast\Z$
since it does not contain $\Z$.   

The initial and final diophantine approximation groups may also be defined intrinsically.   If $R,S\subset\bast \R$ are subgroups, let
\begin{align}\label{RSgroups}
 [R|S](\uptheta ) := \{\bast r\in R|\; \exists\; \bast r^{\perp}\in S\text{ such that }\bast r\uptheta \simeq\bast r^{\perp}\}. 
 \end{align}
The set of $\bast r^{\perp}$ forms a group $[R|S](\uptheta )^{\perp}$ that may be identified with $[S|R](\uptheta^{-1})$
for $\uptheta\not=0$.   For $R=S=\bast \Z$ we recover $\bast\Z (\uptheta )$.  In addition
\begin{align}\label{intrinsicmultideal}
\bast\Z (n\uptheta ) = \left[\bast\Z|\bast\Z[1/n]\right](\uptheta ) \quad \text{and}\quad
\bast (n\Z )(\uptheta )=[\bast(n\Z) |\bast\Z](\uptheta).
\end{align}

 Consider the $\Q$-algebra
\[  \bast\Z_{\Q}:=\underset{\longrightarrow}{\lim}\bast\Z [1/n]  \subset\bast\Q .\]

\begin{prop}
$  \bast\Z(\check{\uptheta}) = [\bast\Z|\bast\Z_{\Q}](\uptheta ) $.
 \end{prop}
 
 \begin{proof} $\bast n\in \bast\Z(\check{\uptheta})$ $\Leftrightarrow$ there exists $q\in\Q$ and $\bast m\in\bast\Z$ with
 $\bast n(q\uptheta )\simeq \bast m$ $\Leftrightarrow$ there exists $\bast\upalpha\in \bast\Z_{\Q}$ 
 with $\bast n\uptheta\simeq \bast \upalpha$ $\Leftrightarrow$ $\bast n\in  [\bast\Z|\bast\Z_{\Q}](\uptheta )$.
 \end{proof}

 Let $\hat{\Z}$ be the profinite completion of $\Z$,
 and denote by $\bast\hat{\Z}_{\upvarepsilon}<\bast\hat{\Z}$ the subgroup 
 of infinitesimals in its ultrapower.
   Let \[ \bast \Z_{\hat{\upvarepsilon}} = \bast\hat{\Z}_{\upvarepsilon}\cap\bast\Z = \bigcap_{n\in\N}\bast (n\Z ),\]  
 the $\Q$-vector space of elements of $\bast\Z$ divisible by any standard integer.
Define the {\bf {\small profinite diophantine approximation group}} of $\uptheta$ as
\[ \bast \Z_{\hat{\upvarepsilon}}(\uptheta ) :=\bast \Z_{\hat{\upvarepsilon}} \cap \bast\Z (\uptheta ) = \bigcap_{n\in\N} \bast (n\Z)(\uptheta )	.\] 
 Notice that 
 $  \bast \Z_{\hat{\upvarepsilon}}(\uptheta ) = [\bast \Z_{\hat{\upvarepsilon}} | \bast \Z](\uptheta )$. 
 
 \begin{prop} $ \bast\Z(\hat{\uptheta}) =[ \bast \Z|\bast \Z_{\hat{\upvarepsilon}} ](\uptheta )$.  In particular, the image of $ \bast\Z(\hat{\uptheta})$ by the duality map $\perp_{\uptheta}:\bast\Z (\uptheta )\rightarrow \bast\Z (\uptheta^{-1})$ 
 is $\bast\Z_{\hat{\upvarepsilon}}(\uptheta^{-1})= [\bast \Z_{\hat{\upvarepsilon}} | \bast \Z](\uptheta^{-1} )$.  
 \end{prop} 
 \begin{proof} $\bast n\in \bast\Z(\hat{\uptheta})$ $\Leftrightarrow$ $\bast n^{\perp}$ is divisible by all standard
natural numbers.  Indeed, let $\bast n_{N}^{\perp}$ be the dual of $\bast n$ viewed as an element of $\bast\Z (\uptheta /N)$.
Then $\bast n(\uptheta /N)-\bast n^{\perp}_{N}\simeq 0$ for all $N$ $\Leftrightarrow$ $\bast n^{\perp} = \bast n^{\perp}_{N}\cdot N$.
Thus the set of such $\bast n^{\perp}$ is contained in $\bast\Z_{\hat{\upvarepsilon}}(\uptheta^{-1})$.  On the other hand,
if $\bast m\in \bast\Z_{\hat{\upvarepsilon}}(\uptheta^{-1})$, then clearly $\bast m^{\perp}\in \bast\Z(\hat{\uptheta})$.
\end{proof}

 The next result shows in particular that $\bast\Z (\hat{\uptheta })$ and $\Z_{\hat{\upvarepsilon}}(\uptheta )$ are distinct.
 
 \begin{theo} Let $\uptheta\in\R -\Q$.  There exist $\bast p\in\bast\Z (\hat{\uptheta })$
 which are not divisible by any $N\in\N$.  In particular $\bast\Z (\hat{\uptheta })-\Z_{\hat{\upvarepsilon}}(\uptheta )$ and $\Z_{\hat{\upvarepsilon}}(\uptheta )-\bast\Z (\hat{\uptheta })$
 are non empty.
 \end{theo}
 
 \begin{proof} Since the projection of $\uptheta\Z$ to $\T$ is dense, there exists, for each positive integer $k$, an integer
$n_{k}$ such that $\|\uptheta/k+n_{k}\uptheta \| < 1/k^{2}$, where $\|\cdot\|$ denotes the distance to the nearest integer. Then the nonstandard integer $\bast p$ associated to the sequence 
$\{ 1 + k!n_{k!}\}_{k\geq1}$  belongs to $\bast\Z (\hat{\uptheta })$ and is not divisible by any prime.  
 Applying the above argument to dual groups shows that  $\Z_{\hat{\upvarepsilon}}(\uptheta )-\bast\Z (\hat{\uptheta })\not=\emptyset$
 as well.
\end{proof}
 
 Note that elements of $ \bast\Z(\hat{\uptheta})$, $\bast\Z(\check{\uptheta})$ no longer have well-defined duals: indeed
when we view $\bast n\in\bast\Z (\uptheta)$ in  $\bast \Z(N\uptheta)$ via (\ref{multiplemapincl}) we have $\bast n^{\perp_{N\uptheta}}=
N\bast n^{\perp_{\uptheta}}$.  In particular, there is no longer an action of ${\rm PGL}_{2}(\Z )$ defined on initial and final
groups.  On the other hand we can consider the numerator denominator final group:
\[ \bast\Z^{1,1}(\check{\uptheta}) = \bigg\langle\bigcup_{a(x)\in {\sf Aff}(\Q)} \bast \Z^{1,1}(a(\uptheta ) )   \bigg\rangle\]
where $\langle X\rangle$ means the group generated by $X$; we were not able to define this group as a direct limit since the 
collection
$\{\bast \Z^{1,1}(a(\uptheta ) ) \} $ is not directed.  

\begin{prop}  $\bast\Z^{1,1}(\check{\uptheta})\not=\bast\Z^{2}$. If $\uptheta$ is a quadratic irrationality then 
$\bast\Z^{1,1}(\check{\uptheta})$ is invariant with respect to the natural action of ${\sf GL}_{2}(\Z )$ on $\bast\Z^{2}$.
\end{prop}

\begin{proof} The first statement follows from the fact that the projection of $\bast\Z^{1,1}(\check{\uptheta})$
onto the first factor is $\bast\Z(\hat{\uptheta})\not=\bast\Z$. 
If $\uptheta$ is a quadratic irrationality then
for any $A\in{\sf PGL}_{2}(\Z )$ we have
$A(\uptheta ) = a_{A}(\uptheta)$ for some $a_{A}(x)\in {\sf Aff}(\Q )$.  Then 
\[  A(\bast \Z^{1,1} (\uptheta ) )= \bast \Z ^{1,1}(A(\uptheta )) = \bast \Z ^{1,1}(a_{A}(\uptheta )) \subset\bast\Z^{1,1}(\check{\uptheta}).\]
For the same reason $A(\bast \Z^{1,1} (a(\uptheta ) ))  \subset\bast\Z^{1,1}(\check{\uptheta})$ for any $a(x)\in {\sf Aff}(\Q )$.
\end{proof}

\begin{note} The group \[ \bast\Z^{1,1}(\hat{\uptheta}):=\bigcap _{a(x)\in {\sf Aff}(\Q)} \bast \Z^{1,1}(a(\uptheta ) )\] is trivial since
for any $\uptheta\not=\upeta$, $\bast\Z^{1,1}(\uptheta)\cap \bast\Z^{1,1}(\uptheta)=(0,0)$.
\end{note}










\section{Residue Class Groups and Kronecker Foliations}\label{ResClassGps}

If $\uptheta=a/b\in\Q$, we have $\bast\Z/\bast\Z (\uptheta )=
 \bast\Z/\bast (b)\cong \Z/(b)$, the residue class field
 of $(b)$ when $b$ is a prime.  Similarly, in the $p$-adic setting (see {\it Nonarchimedean Note} \ref{padic}) the quotient $\bast\Z/\bast\Z (\upxi )$ is isomorphic to 
  $\Z/p^{m}\Z$.  The analogue for diophantine approximation groups of irrationals is given by the following  

\begin{theo}\label{circprop}
Let $\uptheta\in\R-\Q$.  Then $\bast\Z/\bast\Z (\uptheta )$ is an abelian group isomorphic to the circle $\T=\R/\Z$. 
The quotient $\bast\Z/(\bast\Z (\uptheta )+\Z)$ is isomorphic to the ``quantum torus'' $\T_{\uptheta}=\R/\langle 1,\uptheta \rangle$. 
\end{theo}

\begin{proof}  Let $\bast n\in\bast\Z$ and denote by $\overline{\uptheta\bast n}$
the projection of $\uptheta\bast n\in\bast\R$ to $\bast\R/\bast\Z\cong\bast\T$ = the ultrapower of
$\T$.   Since $\T$ is compact, the standard part epimorphism (\ref{standardpartepi})
induces an epimorphism ${\sf std}: \bast\T\rightarrow \T$.
We then define a homomorphism
\begin{align}\label{circlestandmap}  {\sf s}_{\uptheta}:\bast\Z\rightarrow\T,\quad  {\sf s}_{\uptheta}(\bast n )= {\sf std}(\overline{\uptheta\bast n}). 
\end{align}
Since $\uptheta$ is irrational, ${\sf s}_{\uptheta}$ is surjective and has kernel $\bast\Z (\uptheta )$.  If we further quotient
by $\Z$, the result is isomorphic to $\T/{\sf s}_{\uptheta}(\Z )=\T_{\uptheta}$.
\end{proof}

\begin{note} The quantum torus $\T_{\uptheta}$ plays a central role in the Real Multiplication program of Yu. Manin  e.g. see \cite{Man}, \cite{Mar}, \cite{CaGe1}.
\end{note}

Let $\hat{\T} =\underset{\longleftarrow}{\lim} \;\R/n\Z$ be the classical 1-dimensional
solenoid.  Consider also
\[  \check{\T} := \R/(\Q, +) =  \underset{\longrightarrow}{\lim} \;\R/(\Z[1/n]),\] 
 which we call the 
1-dimensional
{\bf {\small dionelos}}.  The dionelos is a non-commutative space (i.e.\ a non Hausdorff quotient that may be studied using
the tools of Noncommutative Geometry) which may
be identified with the leaf space of the adelic lamination 
$ (\A \times \T)/(\Q , + )$.

\begin{coro}\label{resclassgroupsoldio}  $\bast\Z/\bast\Z(\hat{\uptheta}) \cong \hat{\T}$ and  $\bast\Z/\bast\Z(\check{\uptheta}) \cong  \check{\T}$.
\end{coro}

\begin{proof}  By Theorem \ref{circprop}, for all $q\in\Q^{\times}$, $\bast\Z/\bast\Z(q\uptheta)\cong \T$.  Then
$ \bast\Z/\bast\Z (\hat{\uptheta})\cong\underset{\longleftarrow}{\lim} \bast\Z/\bast\Z(q\uptheta) \cong  \hat{\T}$
and $ \bast\Z/\bast\Z(\check{\uptheta}) \cong  \underset{\longrightarrow}{\lim}\bast\Z/\bast\Z(q\uptheta)\cong  \check{\T}$.
\end{proof}

An alternate proof of Corollary \ref{resclassgroupsoldio} in the solenoid case can be made using a standard part map
 $\bast\Z\rightarrow \hat{\T}$, $\bast n\mapsto {\sf std}(\bast n\uptheta )$, where
the latter is defined by taking a $\mathfrak{u}$-recognized limit point of any representative of $\bast n\uptheta$
in $\R\hookrightarrow \hat{\T}$.  Since $\hat{\T}$ is compact this is well defined and gives an epimorphism.


 By Theorem \ref{circprop}, $\bast\Z/\bast\Z (\uptheta )$ does not distinguish $\uptheta\in\R-\Q$, so 
it does not offer a satisfying generalization of residue class field. 
We could instead consider the quotient $\T_{\uptheta}=\bast\Z/(\bast\Z (\uptheta )+\Z)$, 
which does distinguish inequivalent $\uptheta\in\R-\Q$, however
this is still problematic as $\T_{\uptheta}$ is non Hausdorff.  
Our solution will be to regard instead the quotient 
\[ \hat{\R}/\bast\Z (\uptheta ):=(\bbull\R, \R, \bast\Z )/\bast\Z (\uptheta )\]
as providing the sought after analogue notion. 

To justify this choice, we recall that the group $\bast\Z(\uptheta)$ first appeared in \cite{Ge1}
as the {\bf {\small fundamental germ}} 
\[ \llbracket \uppi\rrbracket_{1} \mathfrak{F}(\uptheta ),\] where $\mathfrak{F}(\uptheta )$ is the Kronecker foliation
of the torus by lines of slope $\uptheta$.  The fundamental germ is an analogue of the fundamental group
available for laminated spaces.  

In the paragraphs which follow, we explain this interpretation
in the context of the Kronecker foliations.
In this connection, it will be convenient to call attention to three constructions which can taken as
the definition of, or determine canonically foliations isomorphic to,
the Kronecker foliation of $\uptheta$.  We will use the notation 
$\mathfrak{F}(\uptheta )$ to denote all of them.  

\begin{itemize}
\item[K1.] $\mathfrak{F}(\uptheta )=(\T^{2},\mathcal{L}(\uptheta ))$ where $\mathcal{L}(\uptheta )$
is the image in $\T^{2}$ of the family of lines of slope $\uptheta$ in $\R^{2}$.  This is the usual definition of the Kronecker foliation.
\item[K2.]\label{1param} $\mathfrak{F}(\uptheta )=(\T^{2},L(\uptheta ))$  
where $L(\uptheta)$ is the leaf of $\mathcal{L}(\uptheta )$ passing
through $0$: a 1-parameter subgroup of $\T^{2}$.  The construction K1 is recovered from this pair by considering
the foliation of $\T^{2}$ by the cosets of $L(\uptheta)$.
 \item[K3.]\label{sus} Let $\Z$ act on $\T$ by homeomorphisms via \[ n\cdot\bar{x} := \overline{x+n\uptheta }.\]  Then
$\mathfrak{F}(\uptheta )$ is the quotient of the product foliation $\R\times\T$ by the diagonal action of $\Z$,
\[ n\cdot (r, \bar{x}) := (r-n, n\cdot\bar{x}).\]  The image of $\{ 0\}\times\T$, denoted $\T$, gives a distinguished fiber transversal through $0$.
\end{itemize}

We wish to view $\bbull\R$ as the foliation universal cover for $\mathfrak{F}(\uptheta )$ associated to the fundamental 
germ $\llbracket \uppi\rrbracket_{1} \mathfrak{F}(\uptheta )=\bast\Z(\uptheta)$.  The leaves of this structure
will be the
cosets $\bbull r+\R$, each equipped with the topology of $\R$. 
What is lacking on $\bbull\R$ is
an appropriate transverse topology, a request
complicated by the fact its selection
must break the linear order of $\bbull\R$ (since it is to be compatible with a mapping of $\bbull\R$
{\it onto} the surface $\T^{2}$).  In fact, for the purposes of this paper, it is {\it not necessary} to specify any transverse
topology at all; we do so in order to bring as closely as possible the present constructions into analogy with
classical covering space theory.

 There are many ways of defining a transverse topology, none of them canonical; we shall
consider a special class of topologies associated to a compact subgroup $T<\T^{2}$
transverse to $\mathfrak{F}(\uptheta )$.
The purpose of the transverse topology is to allow us to define 
quotient foliation structures;
it will not make of $\bbull\R$ a foliated topological group but will be preserved by enough algebraic
structure so that upon quotient by $\bast\Z (\uptheta )$ the result will be isomorphic to the foliated group $\mathfrak{F}(\uptheta )$.

We begin by noting that as an abelian group $\bbull\R$ can be identified 
with the suspension
\begin{align}\label{suspofbullR}
\bbull\R =(\R\times\bast\Z )/\Z,\quad n\cdot(r,\bast n) = (r-n, \bast n+n). 
\end{align}
We will want to define a topology $\bast\Z$ so that the action of $\Z$ by translations is by homeomorphisms
acting properly and discontinuously.

\begin{lemm}\label{sectionofspseq}  Let $\T$ be given the $\Z$-set structure described in ${\rm K\ref{sus}}$.
Then the exact sequence of pointed $\Z$-sets associated to the epimorphism (\ref{circlestandmap})
\[  0\rightarrow \bast\Z (\uptheta )\hookrightarrow \bast\Z \stackrel{{\sf s}_{\uptheta}}{\twoheadrightarrow} \T\rightarrow 0 \]
splits by a basepoint preserving $\Z$-set section $\upsigma :  \T\hookrightarrow\bast\Z$.
\end{lemm}

\begin{note} The fact that $\upsigma$ is a basepoint preserving $\Z$-set map means in particular that $\Z\subset \bast T:= \upsigma (\T)$.
\end{note}

\begin{proof} Denote as in K3 the $\Z$-action on $\T$ by $n\cdot\bar{x}$.  Define first $\upsigma (0)=0$ and $\upsigma (n\cdot\bar{0})=n$. This defines
$\upsigma$ over $\Z\cdot \bar{\uptheta}$, well-defined since ${\rm Ker}({\sf s}_{\uptheta})=\bast\Z(\uptheta )$ and $\Z\cap\bast\Z (\uptheta )=\{ 0\}$.  Now given $\bar{x}\in\T$, choose $\upsigma (\bar{x})\in{\sf s}_{\uptheta}^{-1}(\bar{x})$
and define $\upsigma (n\cdot\bar{x})= \upsigma (\bar{x})+n$, so that the set $\upsigma(\bar{x})+\Z$ maps to $\Z\cdot\bar{x}=\bar{x}+\uptheta\Z\mod\Z$.   If $\bar{y}\not\in \Z\cdot\bar{x}$,
any choice $\upsigma(\bar{y})\in {\sf s}_{\uptheta}^{-1}(\bar{y})$ satisfies $(\upsigma(\bar{y}) + \Z)\cap (\upsigma(\bar{x}) + \Z)=\emptyset$.  Thus we can
consistently lift $\Z$ translates of elements of $\T$ to $\Z$ translates of elements in $\bast\Z$.  
In particular, the desired section $\upsigma$ is defined by selecting a representative $\bar{x}$ from each coset $\Z\cdot \bar{x}$, an element
$\upsigma(\bar{x})\in {\sf s}_{\uptheta}^{-1}(\bar{x})$ and extending by the action of $\Z$.\end{proof}

Since $\bast T:=\upsigma (\T)$ is the image of a section to ${\sf s}_{\uptheta}$, the function 
\begin{align}\label{prodstruc*Z}
\bast T\times\bast\Z (\uptheta )\longrightarrow \bast\Z,\quad (\bast t,\bast n)\mapsto \bast t+\bast n,\end{align} is a bijection.
(Indeed if $\bast t+\bast n =\bast t'+\bast n'$ then ${\sf s}_{\uptheta}(\bast t)={\sf s}_{\uptheta}(\bast t')$ implying $\bast t=\bast t'$ and therefore $\bast n=\bast n'$.) 
Give $\bast T$ the topology of $\T$, $\bast\Z (\uptheta )$ the discrete topology and $\bast\Z $ the product topology.
Call this topology $\uptau_{\uptheta,\upsigma}$.
Finally, give $\R\times (\bast\Z,\uptau_{\uptheta, \upsigma})$
the product topology, regarding it as a product foliation with model transversal $(\bast\Z,\uptau_{\uptheta, \upsigma})$, and give $\bbull\R$ the quotient topology by way of  (\ref{suspofbullR}).   We denote the latter topology also by $\uptau_{\uptheta,\upsigma}$.

\begin{prop}\label{folforbullR}  $(\bbull\R, \uptau_{\uptheta, \upsigma})$ has the structure of a (non second countable) foliation.  
\end{prop}


\begin{proof}  The action of $n\in\Z$ on $\bast\Z$ by translation, $\bast m\mapsto \bast m+n$, defines a homeomorphism of $(\bast\Z,\uptau_{\uptheta})$.  Indeed,
 writing $\bast m = (\bast  t,\bast n)\in \bast T\times\bast\Z (\uptheta )$, we see that $\bast m +n = (\bast t+n,\bast n)$, so that 
translation by $n$ takes place entirely along the $\bast T$ factor, where it acts homeomorphically.  
 It follows that the diagonal action of $\Z$ used to define the suspension structure (\ref{suspofbullR}) 
is by homeomorphisms that preserve leaves and transversals, and is properly discontinuous.  Thus $\bbull\R$ acquires
the structure of a (non second countable) foliation.
\end{proof}

An isomorphism of Kronecker foliations $\mathfrak{F}(\uptheta )\rightarrow\mathfrak{F}(\upeta )$ -- which exists $\Leftrightarrow$ 
$\uptheta$ and $\upeta$ are equivalent reals -- will not induce an isomorphism 
of $(\bbull\R, \uptau_{\uptheta, \upsigma})\rightarrow (\bbull\R, \uptau_{\upeta, \upsigma'})$ for appropriate choices of $\upsigma,\upsigma'$, due the fact
 that the transversal $\T$ is transported.
This situation is remedied by considering transverse topologies obtained from an arbitrary compact subgroup of 
$\T^{2}$ isomorphic to $\T$.  Such a subgroup is the image in $\T$ of a line with rational slope $q\in\Q$, so we write $\T_{q}$;
note that $\T=\T_{\infty}$.
Consider the discrete group 
\[ \Z_{q}:=\{ r\in\R |\; \overline{(r,r\uptheta )}\in \T_{q}\}\cong\Z\] (where
$\overline{(x,y)}$ is the image of $(x,y)$ in $\T^{2}$).  Since $\Z_{q}$ is discrete, its ultrapower $\bast \Z_{q}\subset\bast\R$
survives the quotient by infinitesimality to define a subgroup of $\bbull\R$ and 
\[ \bbull\R=(\R\times \bast \Z_{q})/\Z_{q}.\]  Notice also that if $\bast r\in\bast \Z _{q}$ satisfies ${\rm std}(\bast r )=\bar{0}$ then $\bast r\simeq \bast n$ for some
$\bast n\in\bast\Z (\uptheta )$: that is, viewed in $\bbull\R$, $\bast r\in\bast\Z (\uptheta )$, so that
$\bast\Z_{q}$ is an extension of $\bast\Z (\uptheta )$.  The standard part map maps $\bast\Z _{q}$ epimorphically onto $\T_{q}$ with kernel 
$\bast \Z (\uptheta )$
and proceeding as above we may define a topology $\uptau_{\uptheta,\upsigma,q }$.  

Now if $\upeta =A(\uptheta )$ where 
$A=\left(\begin{array}{cc}
a & b \\
c & d
\end{array}\right)$ then the group automorphism $h_{A}:\bbull\R\rightarrow\bbull\R$, $\bbull r\mapsto (c\uptheta+d )\bbull r$, maps $\bast\Z (\uptheta )$
isomorphically onto $\bast\Z (\upeta )$: for all $\bast n\in \bast\Z (\uptheta )$, $(c\uptheta+d )\bast n = c\bast n^{\perp}+d\bast n\in\bast\Z$
and $\upeta(c\uptheta+d )\bast n=a\bast n^{\perp}+b\bast n\in \bast\Z$.  Then $h_{A}$
induces a leaf preserving bijection $\mathfrak{F}(\uptheta )\rightarrow \mathfrak{F}(\upeta )$.
The latter is induced by the linear map of $\C$, $A^{\ddag} := \left(\begin{array}{cc}
d & c \\
b & a
\end{array}\right)$ which maps $\T$ onto $\T_{q}$ for $q=a/c$.  It follows that there exists a section $\upsigma'$ so that
the map $h_{A}:(\bbull\R, \uptau_{\uptheta,\upsigma,\infty})\rightarrow (\bbull\R, \uptau_{\upeta,\upsigma',a/c})$ is a homeomorphism.
We summarize these observations by saying that for each $\uptheta$ there is a canonical {\it topology set} 
\[ {\rm Top}(\uptheta ) = 
\{\uptau_{\uptheta,\upsigma,q}\},\] and any equivalence $A(\uptheta )=\upeta$ induces a bijection 
$A:{\rm Top}(\uptheta )\rightarrow {\rm Top}(\upeta )$ of homeomorphic topologies.  We refer to any element of
${\rm Top}(\uptheta )$ as a {\bf {\small section topology}}.


In what follows we denote $\hat{\R}_{q} = (\bbull\R, \R ,\bast\Z_{q})$, and when $q=\infty$ we write simply $\hat{\R}$.

\begin{theo}\label{kron}  The group $\bast\Z(\uptheta)\subset(\bbull\R,\uptau_{\uptheta,\upsigma,q})$
acts properly discontinuously by homeomorphisms which permute isometrically the leaves of $\bbull\R$ and preserve the distinguished transversal $\bast\Z_{q}$. 
The quotient $\hat{\R}_{q}/\bast\Z(\uptheta)$
is a foliation with distinguished transversal canonically isomorphic to $(\mathfrak{F}(\uptheta ),\T_{q})=(\T^{2}, L(\uptheta), \T_{q})$.
\end{theo}

\begin{proof}  Define an epimorphism ${\sf std}_{\uptheta}:\bbull\R \rightarrow\T^{2}$ as follows: for $\bbull r\in\bbull\R$, choose
a representative sequence $\{ r_{i}\}$, map it to the line of slope $\uptheta$,  $\{(r_{i},\uptheta r_{i})\}$, then
project the latter to $\T^{2}$.  Since $\T^{2}$ is compact, the image of $\{(r_{i},\uptheta r_{i})\}$ will have limit
points and the ultrafilter $\mathfrak{u}$ will recognize exactly one of them: call it ${\sf std}_{\uptheta} (\bbull r)$.
Note that ${\sf std}_{\uptheta} (\bbull r)$ does not depend on the choice of $\{ r_{i}\}$ since a sequence which
is infinitesimal to it will produce the same limit points.  We have that ${\sf std}_{\uptheta}(\R)=L(\uptheta )$, and that in general, the leaf $\bbull r+\R$ is mapped to the leaf of $\mathfrak{F}(\uptheta )$
through ${\sf std}_{\uptheta} (\bbull r)$.  Moreover,  ${\sf std}_{\uptheta}(\bast\Z_{q})=\T_{q}$.  
The kernel of ${\sf std}_{\uptheta}$ is $\bast\Z(\uptheta )$ and ${\sf std}_{\uptheta}$ induces a topological isomorphism 
$(\bbull\R_{q}/\bast\Z(\uptheta ), R)\cong(\T^{2}, L(\uptheta ))$ where $R$ is the image of $\R$ in $\bbull\R/\bast\Z(\uptheta )$. 
\end{proof}

 We will generally drop the reference to the section topology $\uptau_{\uptheta,\upsigma,q}$ and write simply $\bbull\R$ if no confusion occurs.
We regard the triple \[ \hat{\R}_{q} = (\bbull\R, \R,\bast\Z_{q} )\]
as determining the foliation of $\bbull\R$ (the cosets of the subgroup $\R$) along with the distinguished transversal $\bast\Z_{q}$.  
The quotient $\hat{\R}/\bast\Z (\uptheta )$ may be viewed as the residue class group of ``$\uptheta$-roots of 1'', in which the action of $\R^{\times}_{+}$ provided by flow along the leaves is akin to a real Frobenius map.  Note that the return maps of this flow to the transversal 
 $\T$ define the irrational rotation by $\uptheta$.  This will be pursued in greater detail in \cite{Ge6}.

 We briefly discuss the extremal situation. Recall that for each integer
$n$, $\bast\Z (\uptheta/n )\subset\bast\Z (\uptheta )$.  This inclusion induces a covering
of Kronecker foliations \[ \mathcal{F}(\uptheta/n )\longrightarrow\mathcal{F}(\uptheta )\] (a covering $\T^{2}\rightarrow \T^{2}$ mapping
$L(\uptheta/n)$ to $L(\uptheta)$), defined $(x,y)\mapsto (x, ny )$.  This covering corresponds to opening the transversal
$\T$ by a factor of $n$.
Denote \[ \mathfrak{F}(\hat{\uptheta} )= \underset{\longleftarrow}{\lim} \mathfrak{F}(\uptheta/n),\]  
a solenoid isomorphic to
  $\T\times \hat{\T}$.   The solenoid leaves of $ \mathfrak{F}(\hat{\uptheta} )$ are isomorphic
  to infinite cylinders each of which is equipped with a linear foliation. 
We may also consider the direct limit \[ \mathfrak{F}(\check{\uptheta} ):= \underset{\longrightarrow}{\lim} \mathfrak{F}(n\uptheta )\]
defined by the maps $(x,y)\mapsto (x, y/n )$, 
 a non-Hausdorff group isomorphic
 to $\T\times  \check{\T}$, ``foliated'' by the cosets of the direct limit group 
$ L(\check{\uptheta}):=\underset{\longrightarrow}{\lim}\; L(n\uptheta )$.  We have the following analogue
 of Theorem \ref{kron}:
 
 \begin{theo}\label{uniformize}  There exist topologies on $\bbull\R$ so that 
 \[ \bbull\hat{\R}/\bast\Z (\hat{\uptheta}) \cong (\mathfrak{F}(\hat{\uptheta} ), \hat{\T}), \quad
 \bbull	\hat{\R}/\bast\Z (\check{\uptheta})
  \cong ( \mathfrak{F}(\check{\uptheta} ) , \check{\T}) .\]
 \end{theo}
 
 \begin{proof}  Essentially the same argument as that appearing in Theorem \ref{kron}, with the following alteration: the topology on $\bast\Z$ is defined using $\Z$-set sections $\hat{\upsigma},\check{\upsigma}$ to the projections $\bast\Z\rightarrow\hat{\T}^{1},  \check{\T}^{1}$.
 \end{proof}
 

\begin{note}[Model Theoretic Interlude] From the point of view of model theory, structures allied to
the dionelos Kronecker foliation  $ \mathfrak{F}(\check{\uptheta} )$ are conjecturally better behaved at the level of ranking  than  
counterparts related to $ \mathfrak{F}(\uptheta )$. Consider the {\it green field} \cite{Zi1} 
\[ (\C, +,\times, G_{\uptau} (\Q ) )\] where
 the predicate $G_{\uptau} (\Q )$ is defined
 $G_{\uptau} (\Q ) = \{  \exp (\uptau t + q) |\; t\in\R , \; q\in\Q  \} < \C^{\ast}$ for $\uptau = 1+i. $
Then
\[   \C^{\ast}/G_{\uptau} (\Q ) \cong \T(2\check{\uppi}):=  \check{\T}/\langle
2\overline{\uppi}\rangle,\] 
the dionelos quantum torus associated to $2\uppi$ i.e.\ the leaf space of  $\mathfrak{F}(2\check{\uppi} )$ \cite{BaGe}. In \cite{Zi1}, \cite{CaZi} it has been shown --  assuming the Schanuel conjecture -- that the green field is $\upomega$-stable, which allows one to
rank definable sets and obtain in theory a notion of algebraic geometry based upon it \cite{Zi3}.  By contrast, the quantum torus 
$ \T_{2\uppi }$ (the leaf space of $\mathfrak{F}(2\uppi )$) 
corresponds to the {\it emerald field} $(\C, +,\times, G_{\uptau} (\Z ) )$, which satisfies the weaker notion of super stability \cite{Zi2} and permits only the ranking of types (ideals).   See the Appendix for a discussion of how diophantine approximation groups
may be viewed as types.
\end{note}

In \S \ref{daringint} we will need to consider generalizations of Kronecker foliations  
associated to the groups 
$ \bast\Z^{1,1}(\uptheta )$. 
Consider the quotient \[ \bbull\R^{2}/\bast\Z^{1,1}(\uptheta ).\] 
If we denote 
\[ \bbull L(\uptheta)   = \{  ( \bbull r, \uptheta\bbull r)|\; \bbull r\in \bbull\R \},\] 
the extended line of slope $\uptheta$,  then $\bast\Z^{1,1}(\uptheta )\subset  \bbull L(\uptheta)$ 
since $(\bast n,\uptheta \bast n)= (\bast n, \bast n^{\perp})$ in $\bbull\R^{2}$.  Then the bijection \[ \bbull\R\longleftrightarrow  \bbull L(\uptheta),\quad
\bbull r\mapsto (\bbull r, \uptheta\bbull r)\] takes $\bast\Z (\uptheta )$ to $\bast\Z^{1,1}(\uptheta )$ and any section topology $\uptau\in{\rm Top}(\uptheta)$
on $\bbull\R$ 
induces one
on $\bbull L(\uptheta)$ giving
\[ \mathfrak{F}(\uptheta )\cong\bbull L(\uptheta)/\bast\Z^{1,1}(\uptheta ).\]   The line
$\bbull L(\uptheta^{-1})= \{  (  \bbull r, \uptheta^{-1}\bbull r)|\; \bbull r\in \bbull\R \}$ is similarly topologized by a section topology $\uptau'\in {\rm Top}(\uptheta^{-1})$, which may be
chosen to correspond to $\uptau$ by a projective
linear equivalence
$A(\uptheta)=\uptheta^{-1}$ (unique if $\uptheta$ is not a quadratic irrationality).
Since $\bbull\R^{2}= \bbull L(\uptheta)\oplus \bbull L(\uptheta^{-1})$ we give it the product topology.  Note that we have
\[ \bbull\R^{2}/\bast\Z^{1,1}(\uptheta )  \cong   \mathfrak{F}(\uptheta ) \oplus\bbull L(\uptheta^{-1})\]
which we may view as a linear foliation by 2-planes with complete transversals homeomorphic to 
$\bast\Z\times\T$.
Denote this foliation $\mathfrak{F}^{1,1}(\uptheta )$, the {\bf {\small extended Kronecker foliation}} of slope $\uptheta$.  



\section{Diophantine Approximation Groups of Real Matrices}\label{higher}

Consider an $r\times s$ matrix $\Uptheta = (\uptheta_{ij})$ with entries in $\R$.  
Then 
the associated group of vector diophantine approximations
\[  \bast\Z^{s}(\Uptheta) = \{ \bast {\boldsymbol n}\in \bast\Z^{s} |\; \exists  \bast {\boldsymbol n}^{\perp}\text{ s.t. }  
\Uptheta\bast {\boldsymbol n}-\bast {\boldsymbol n}^{\perp}=\upvarepsilon (\Uptheta )(\bast {\boldsymbol n})\ \in\bast\R^{r}_{\upvarepsilon} \}\] is called the {\bf {\small inhomogeneous diophantine approximation group}} of $\Uptheta$.
The dual group $\bast\Z^{r}(\Uptheta)^{\perp}$ and the error group
$\bast\R^{r}_{\upvarepsilon}(\Uptheta )$ are defined as in \S \ref{ratapprox}.  If
$r=s$ and
$\Uptheta$ is invertible then $\bast\Z^{r}(\Uptheta)^{\perp}=\bast\Z^{r}(\Uptheta^{-1})$ is isomorphic
to $\bast\Z^{s}(\Uptheta)$.  

In this setting, there is an important subgroup, the {\bf {\small homogeneous diophantine approximation group}}, defined by
\[  \bast\widetilde{\Z}^{s}(\Uptheta) = {\rm Ker}(\perp ) = \{  \bast {\boldsymbol n}\in \bast\Z^{s}(\Uptheta) |\;  \Uptheta \bast{\boldsymbol n}
\in\bast\R^{r}_{\upvarepsilon} \} < \bast\Z^{s}(\Uptheta).\]
We denote by $\bast\widetilde{\R}^{r}_{\upvarepsilon}(\Uptheta )= \big\{\Uptheta \bast{\boldsymbol n} |\; \bast {\boldsymbol n}\in \bast\widetilde{\Z}^{s}(\Uptheta) \big\}$ the corresponding group of homogeneous errors and by $\tilde{\upvarepsilon} (\Uptheta )$ the homogeneous error map.

\begin{theo}\label{homogeneouslinindep}  The columns of $\Uptheta$ are linearly independent over $\Q$ $\Leftrightarrow$ the homogeneous error map
$ \tilde{\upvarepsilon} (\Uptheta ): \bast\widetilde{\Z}^{s}(\Uptheta) \rightarrow \bast\widetilde{\R}^{r}_{\upvarepsilon}(\Uptheta )  $
is an isomorphism $\Leftrightarrow$ $\bast\widetilde{\Z}^{s}(\Uptheta )$ contains no $\bast\Z^{s}$-ideals.
\end{theo}

\begin{proof}  Let $\Uptheta^{(1)},\dots , \Uptheta^{(s)}$ be the columns of $\Uptheta$.  
The kernel of $\tilde{\upvarepsilon} (\Uptheta )$ consists of vectors $\bast {\boldsymbol n}=(\bast n_{1},\dots ,\bast n_{s})$ for which
$  \bast n_{1}\cdot \Uptheta^{(1)} + \cdots + \bast n_{s}\cdot \Uptheta^{(s)} = \text{\bf 0}$,
where $\text{\bf 0}$ is the zero vector.  This occurs $\Leftrightarrow$ there are sequences $\{ n_{1,i}\}, \dots ,\{ n_{s,i}\}$
representing the classes $\bast n_{1},\dots ,\bast n_{s}$ with
$ n_{1,i}\cdot \Uptheta^{(1)} + \cdots + n_{s,i}\cdot \Uptheta^{(s)} = \text{\bf 0}$.   
For $\bast {\boldsymbol n}$ non-zero, this is the case $\Leftrightarrow$ the columns of $\Uptheta$ are dependent over $\Q$.  
For the last bi-implication, we note that ${\rm Ker}( \tilde{\upvarepsilon} (\Uptheta ))\subset \bast\widetilde{\Z}^{s}(\Uptheta )$ is necessarily a $\bast\Z^{s}$-ideal, and that any $\bast\Z^{s}$-ideal contained in $\bast\widetilde{\Z}^{s}(\Uptheta )$ must
be contained in ${\rm Ker}( \tilde{\upvarepsilon} (\Uptheta ))$.
\end{proof}


Let, as in the proof of the Theorem \ref{homogeneouslinindep},  $\Uptheta^{(1)},\dots , \Uptheta^{(s)}$ be the columns of $\Uptheta$
and denote by ${\boldsymbol e}_{1},\dots ,{\boldsymbol e}_{r}$ the standard orthonormal basis of $\R^{r}$.   We say that
the columns of $\Uptheta$ are {\bf {\small inhomogeneously independent}} over $\Q$ if the set 
$\Uptheta^{(1)},\dots , \Uptheta^{(s)}, {\boldsymbol e}_{1},\dots ,{\boldsymbol e}_{r}$ is independent over $\Q$.
The proof of the following Theorem is formally the same as that of Theorem \ref{homogeneouslinindep}:

\begin{theo}\label{inhomoglinindep}   The columns of $\Uptheta$ are inhomogeneously independent over $\Q$
$\Leftrightarrow$ the inhomogeneous error map 
$ \upvarepsilon (\Uptheta ): \bast\Z^{s}(\Uptheta )\rightarrow \bast\R^{r}_{\upvarepsilon}(\Uptheta )$ is an isomorphism $\Leftrightarrow$
$\bast\Z^{s}(\Uptheta )$ contains no $\bast\Z^{s}$-ideals.
\end{theo}

There are a few special cases worth singling out:
\begin{itemize}
\item $s=1$, $\Uptheta = (\uptheta_{1},\dots , \uptheta_{r})^{\sf T}$ is a column vector.
Then $\bast\Z(\Uptheta) \subset\bast\Z$ is the {\bf {\small simultaneous diophantine approximation group}}:
 \begin{align}\label{simultanDA} \bast\Z(\Uptheta)= \bigcap  \bast\Z(\uptheta_{i} ) . \end{align}
 For general $\Uptheta$ of size $r\times s$ we have canonical inclusions $\bast\Z\big(\Uptheta^{(j)}\big)\subset \bast \Z^{s}(\Uptheta )$ of the column
 vector diophantine approximation groups, defined $\bast n\mapsto (0,\dots, \bast n,\dots ,0)$ for $j=1,\dots ,s$.  It follows
 that $\bast \Z^{s}(\Uptheta )$ has the cardinality of the continuum as the column vector groups have this cardinality, by (\ref{simultanDA})
 and Proposition \ref{cardcont}.
 \item $r=1$, $\Uptheta = (\uptheta_{1},\dots , \uptheta_{s})$ is a row vector.  Then Theorem \ref{homogeneouslinindep} 
(Theorem \ref{inhomoglinindep})
provides a criterion for when the coordinates of $\Uptheta$ (the coordinates of $\Uptheta$ and 1) are linearly independent over $\Q$.  This case will be invoked in the rigidity reformulations discussed in \S \ref{rigidity}. 
\item $r=1$, $\Uptheta = (\uptheta_{1},\dots , \uptheta_{s})$ is a row vector.  Then the dual group $\bast\Z(\Uptheta)^{\perp}<\bast\Z$ is the {\bf {\small nonprincipal diophantine approximation group}} generated by the coordinates $\uptheta_{1},\dots \uptheta_{s}$ of 
$\Uptheta$.   Indeed, ${\boldsymbol n}^{\perp}\in \bast\Z(\Uptheta)^{\perp}\subset\bast\Z$ if and only if
\[  \bast{\boldsymbol n}^{\perp} = \uptheta_{1}\bast n_{1}+ \cdots +  \uptheta_{s}\bast n_{s} \mod \bast\R_{\upvarepsilon} \]
where $\bast{\boldsymbol n} = (\bast n_{1}, \dots ,\bast n_{s})\in \bast \Z^{s}(\Uptheta )$.  Notice that for all $i$
in which $\uptheta_{i}\not=0$,
\[  \bast\Z(\Uptheta)^{\perp}\supset\bast\Z(\uptheta_{i})^{\perp}
=\bast\Z(\uptheta_{i}^{-1})  . \]
If $\Uptheta\in \Q^{s}$ and $\uptheta_{i}=c_{i}/d_{i}$ then 
 $ \bast\Z(\Uptheta)^{\perp}=\bast (\gcd (d_{1}, \dots , d_{s}))$.
 \end{itemize}


\begin{theo}\label{matinhomindep}  The rows of $\Uptheta$ are inhomogeneously independent over $\Q$ $\Leftrightarrow$ 
$\bast\Z^{s}/\bast\Z^{s}(\Uptheta )\cong \T^{r}=\R^{r}/\Z^{r}$.
\end{theo} 

\begin{proof}  Define the map $ {\sf s}_{\Uptheta} :\bast\Z^{s}\rightarrow \T^{r}$ by 
$\bast {\boldsymbol n}\mapsto {\sf std}( \overline{\Uptheta \bast{\boldsymbol n}})$,
where $ \overline{\Uptheta \bast{\boldsymbol n}}$ is the image of $\Uptheta \bast{\boldsymbol n}$ in
$\bast\T^{r}$ and
${\sf std}:\bast\T^{r}\rightarrow \T^{r}$
 is the multi-variate 
standard part map.  Then $\T(\Uptheta ):= {\sf s}_{\Uptheta}  (\bast\Z^{s})$ is a closed subgroup of $\T^{r}$ and
${\rm Ker}({\sf s}_{\Uptheta} )=\bast \Z^{s}(\Uptheta )$ by definition of the latter.  In particular $\T(\Uptheta )\cong \bast\Z^{s}/ \bast\Z^{s}(\Uptheta )$. 
Suppose the rows $\left\{ \Uptheta_{(1)},\dots , \Uptheta_{(r)}\right\}$ of $\Uptheta$ and the coordinates 
${\boldsymbol e}_{1},\dots ,{\boldsymbol e}_{r}$ satisfy a relation over $\Q$.  Then they also satisfy a relation over $\Z$,
so we have $a_{1}\Uptheta_{(1)}=a_{2}\Uptheta_{(2)}+\cdots +a_{r}\Uptheta_{(r)}\mod\Z^{n}$ for $a_{i}\in\Z$.  
It follows that the image $\overline{\Uptheta\Z^{s}}$ of $\Uptheta\Z^{s}$
in $\T^{r}$ contains a finite index subgroup contained in the image $\bar{P}\subset\T^{r}$ of the plane 
$P=\{ (a_{2}x_{2}+\cdots +a_{r}x_{r},x_{2},\dots ,x_{r})|\; x_{1},\dots ,x_{r}\in\R\}$.  But the latter is a closed subgroup not equal to $\T^{r}$.  
Since $\T(\Uptheta )$ is the closure of $\overline{\Uptheta\Z^{s}}$,
$\T(\Uptheta )\not=\T^{r}$.  Conversely, if $\T(\Uptheta )\not=\T^{r}$ then the connected component $\T(\Uptheta )_{0}$ of
 $\bar{0}$ is the closure of the image of $\Uptheta(\Upgamma )$ for some lattice $\Upgamma\subset\Z^{s}$. This implies
 that the rows of $\Uptheta$ are inhomogeneously dependent over $\Q$.  
\end{proof}

Theorems  \ref{homogeneouslinindep} and \ref{inhomoglinindep} give the first columns in the $\Q$
independence entry of the dictionary given in the Introduction.  To get the last column,
we associate a pair of Kronecker foliations to $\Uptheta$. Let $\widetilde{L}(\Uptheta)\subset\R^{s+r}$ be the graph of $\Uptheta :\R^{s}\rightarrow\R^{r}$, let
 $L(\Uptheta )$ be the image of $\widetilde{L}(\Uptheta)$ in $\T^{s+r}$
and let \[ \mathfrak{F}(\Uptheta ) = (\T^{s+r}, L(\Uptheta))\] be the associated Kronecker foliation.  
Notice that $\mathfrak{F}(\Uptheta )$ fibers over $\T^{s}$ with transversal $\T^{r}$, and that the set of leaves intersecting
$\T(\Uptheta )$ (defined
in the proof of Theorem \ref{matinhomindep}) defines a minimal subfoliation denoted $\mathfrak{F}_{0}(\Uptheta )$.  We call $\mathfrak{F}(\Uptheta )$ the
{\bf {\small inhomogeneous Kronecker foliation}} associated to $\Uptheta$.

 We may also define a Kronecker foliation of the ``base'' torus $\T^{s}$ as follows.  Let 
 \[ \widetilde{\mathfrak{l}}(\Uptheta)=
\widetilde{L}(\Uptheta)\cap (\R^{s}\times \boldsymbol 0):\] the part of the graph $\widetilde{L}(\Uptheta)$
corresponding to ${\rm Ker}(\Uptheta )$.   Let $\mathfrak{l}(\Uptheta )$ be its image in $\T^{s}$ and   
let 
\[ \mathfrak{f}(\Uptheta):= (\T^{s}, \mathfrak{l}(\Uptheta )),\] the {\bf {\small homogeneous Kronecker foliation}} associated to $\Uptheta$.  Notice that there is a natural projection
\[  \mathfrak{F}(\Uptheta )\longrightarrow \mathfrak{f}(\Uptheta )  \]
induced by $\T^{r+s}\rightarrow \T^{s}$.  

Denote as before
\[ \hat{\R}^{s}=(\bbull\R^{s},  \R^{s},\bast\Z^{s}),\] where $\bbull\R^{s}$ is to be given 
a section topology corresponding to the projection $ {\sf s}_{\Uptheta} :\bast\Z^{s}\rightarrow \T(\Uptheta )$, just as was done
in \S \ref{ResClassGps} for $\bbull\R$.
We say a few words about what this entails: write $\bbull\R^{s}$ as a suspension by $\Z^{s}$ as in (\ref{suspofbullR}).  Now if we
consider $\mathfrak{F}(\Uptheta )$ in its suspension form (in the style of K3 of \S \ref{ResClassGps}) we see that
the fiber $\T^{r}$ acquires through $\Uptheta$ the structure of a $\Z^{s}$ set via $\boldsymbol n\cdot \bar{\boldsymbol x}=
\bar{\boldsymbol x} + \overline{\Uptheta\boldsymbol n}$, and the subgroup $\T(\Uptheta )$ is a sub $\Z^{s}$ set.
A section to $ {\sf s}_{\Uptheta}$ may be constructed
following the recipe of Lemma \ref{sectionofspseq}, and through this section we identify 
$\bast\Z^{s}$ with $ \bast\Z^{s}(\Uptheta )\times \bast T(\Uptheta )$ for $\bast T(\Uptheta )={\sf s}_{\Uptheta} (\T(\Uptheta ))$ .
This gives the topology on the transversal $\bast\Z^{s}$ needed to make $\bbull\R^{s}$ a foliation.  Any subgroup \[ \T_{\boldsymbol q}\subset\T^{s+r}\] isomorphic to $\T^{r}$ (where $\boldsymbol q=(q_{1},\dots ,q_{r-1})$ is the slope) gives rise to a new linear transversal 
$\T_{\boldsymbol q}(\Uptheta )$ with respect to which one can define other sectional topologies, using a suspension
structure based on $\bast\Z^{s}_{\boldsymbol q}(\Uptheta )$ (where $\Z^{s}_{\boldsymbol q}(\Uptheta)=\{\boldsymbol r|\; \overline{(\boldsymbol r, \Uptheta\boldsymbol r)}\in \T_{\boldsymbol q}(\Uptheta )\}$) as in the discussion
found in \S \ref{ResClassGps} after Proposition \ref{folforbullR}.  We denote \[ \hat{\R}^{s}_{\boldsymbol q}=(\bbull\R, \R, 
\T_{\boldsymbol q}(\Uptheta ))\] equipped with a corresponding section topology.


\begin{theo}\label{indeptheo}   Let $\Uptheta$ be an $r\times s$ real matrix.  Then 
\[ \hat{\R}^{s}_{\boldsymbol q}/\bast\Z ^{s}(\Uptheta )\cong 
(\mathfrak{F}_{0}(\Uptheta ), \T(\Uptheta )_{\boldsymbol q}).\]
\end{theo}

\begin{proof}  Essentially the same as Theorem \ref{kron}.  
We define a continuous homomorphism ${\sf std}_{\Uptheta}:\bbull\R^{s} \rightarrow\T^{r+s}$,
mapping a representative $\{ \boldsymbol r_{i}\}\subset\R^{s}$ of the class $\bbull \boldsymbol r$ to $L(\Uptheta)$ by $(\boldsymbol r_{i}, \Uptheta\boldsymbol r_{i})$, then to the unique $\mathfrak{u}$-defined
limit point of its projection to $\T^{r+s}$.   The image of ${\sf std}_{\Uptheta}$ is the underlying manifold of $\mathfrak{F}_{0}(\Uptheta )$,
and we have ${\sf std}_{\Uptheta}(\R^{s})$ = $L(\Uptheta)$, 
 ${\sf std}_{\Uptheta}(\bast\Z_{\boldsymbol q}^{s})=\T(\Uptheta )_{\boldsymbol q}$ and  ${\rm Ker} ({\sf std}_{\Uptheta})=\bast\Z^{s}(\Uptheta )$.
\end{proof} 

The following results fill out the $\Q$-independence dictionary:

\begin{theo}\label{indeptheo}  Let $ \Upgamma(\Uptheta )=\{ (\boldsymbol n, \Uptheta\boldsymbol n)|\; 
\boldsymbol n\in {\rm Ker}( \upvarepsilon (\Uptheta ))\cap\Z^{s}\}\subset\Z^{s+r}$.  Then  
\begin{align}\label{leaftop}
 L(\Uptheta)=\widetilde{L}(\Uptheta)/\Upgamma ( \Uptheta)\cong\R^{s}/ ({\rm Ker}( \upvarepsilon (\Uptheta ))\cap\Z^{s}).
 \end{align} In particular, $\mathfrak{F}(\Uptheta )$
has planar leaves $\Leftrightarrow$ the columns of $\Uptheta$ are inhomogeneously independent over $\Q$.
 \end{theo}
 
 \begin{proof} 
Clearly $\widetilde{L}(\Uptheta)\cap\Z^{s+r} = \Upgamma(\Uptheta )$.   The isomorphism $\R^{s}\rightarrow \widetilde{L}(\Uptheta)$,
$\boldsymbol r\mapsto (\boldsymbol r, \Uptheta\boldsymbol r )$, takes $ ({\rm Ker}( \upvarepsilon (\Uptheta ))\cap\Z^{s})$ to $\Upgamma (\Uptheta )$, giving (\ref{leaftop}).  On the other hand, $\Uptheta \bast \boldsymbol n -\bast \boldsymbol n^{\perp}=\boldsymbol 0$ holds for non trivial $\bast  \boldsymbol n\in\bast\Z^{s}$ $\Leftrightarrow$ there exist representative sequences $\{ \boldsymbol n_{i}\}, \{ \boldsymbol n_{i}^{\perp}\}$ for which $\Uptheta  \boldsymbol n_{i} -\boldsymbol n_{i}^{\perp}=\boldsymbol 0$ $\mathfrak{u}$-eventually, so that
$\{ \Uptheta^{(1)},\dots , \Uptheta^{(s)}, {\boldsymbol e}_{1},\dots ,{\boldsymbol e}_{r}\}$ is $\Q$-dependent $\Leftrightarrow$ 
${\rm Ker}( \upvarepsilon (\Uptheta ))\not=0$ $\Leftrightarrow$  $ ({\rm Ker}( \upvarepsilon (\Uptheta ))\cap\Z^{s})\not=0$
$\Leftrightarrow$ $\mathfrak{F}(\Uptheta )$
has non-trivial topology along the leaves.
 \end{proof}
 
The topological status of the leaves of the homogeneous foliation $\mathfrak{f}(\Uptheta )$
 determines homogeneous dependence of the columns of $\Uptheta$:
 
 \begin{theo}\label{indeptheohomog}  The homogeneous foliation $\mathfrak{f}(\Uptheta )$ has planar leaves $\Leftrightarrow$  the columns of $\Uptheta$ are linearly independent over $\Q$.
 \end{theo}
 
 \begin{proof}  The proof is the same as that of Theorem \ref{indeptheo}, only one replaces $\Upgamma (\Uptheta )$ by
  \[ \upgamma(\Uptheta )=\{ (\boldsymbol n, \boldsymbol 0)|\; 
\boldsymbol n\in {\rm Ker}( \tilde{\upvarepsilon} (\Uptheta ))\cap\Z^{s}\}\subset\Z^{s}\times \boldsymbol 0\]
and observes that $\mathfrak{l}(\Uptheta ) = \tilde{\mathfrak{l}}(\Uptheta )/\upgamma(\Uptheta )$.
 \end{proof}
 




\vspace{3mm}

\begin{note} Since inhomogeneous linear independence $\Rightarrow$ homogeneous linear independence, trivial leaf-wise topology in $\mathfrak{F}(\Uptheta )$ implies the same for $\mathfrak{f}(\Uptheta )$.  In addition, inhomogeneous
independence of the rows of $\Uptheta$ implies that the transversal $T(\Uptheta)=\T^{r}$ i.e. that $\mathfrak{F}(\Uptheta )$
is {\it minimal}.
\end{note}

 \section{Diophantine Approximation Groups in Rings of Integers}\label{daringint}

In this section, we re-consider the constructions of \S\S \ref{ratapprox} -- \ref{higher}, replacing $\Q$ by a finite algebraic extension $K/\Q$ of degree $d$.   
Let $\mathcal{O} = \mathcal{O}_{K}$ be the
ring of $K$-integers.   We write elements of $\C^{d}$ in the form ${\boldsymbol z}=(z_{\upnu})$
where the coordinates are indexed by the $d$ infinite places $\upnu :K\hookrightarrow \C$.
Consider the $\R$-algebra  
\[{\K} = \{\boldsymbol z\in \C^{d}|\; \bar{z}_{\upnu} = z_{\bar{\upnu}}\} \cong \R^{r}\times\C^{s}\]
where $r$ = the number of real places and $2s$ = the complex places of $K$, equipped with its euclidean
norm $\|\cdot \|_{\K}$ (the restriction of the hermitian norm in $\C^{d}$).
Recall that $K$ embeds in ${\K}$ via the place embeddings as a subfield and the quotient \[\T_{K}={\K}/\mathcal{O}\] is the Minkowski torus \cite{Ne}.

Define $\bast\C$, $\bast\C_{\rm fin}$, $\bast\C_{\upvarepsilon}$ and $\bbull\C$ as in \S \ref{ratapprox}. Fix the ultrapowers $\bast\mathcal{O}\subset \bast K\subset \bast {\K}$, and denote by $\bast {\K}_{\upvarepsilon}\cong
\bast\R^{r}_{\upvarepsilon}\times\bast\C^{s}_{\upvarepsilon}$
the subgroup of infinitesimals.  We consider $\uptheta\in\R$ as an element of
$\bast {\K}$ via the diagonal embedding i.e.\ as the element $(\uptheta, \dots ,\uptheta )$.
An {\bf {\small $\mathcal{O}$-diophantine approximation}} of $\uptheta\in\R$ is an element
$\bast\upalpha\in\bast\mathcal{O}$ for which there exists $\bast\upalpha^{\perp}\in \bast\mathcal{O}$
with
\begin{equation}\label{ODA}
   \uptheta\bast\upalpha - \bast\upalpha^{\perp}\in  \bast {\K}_{\upvarepsilon}.   
   \end{equation}

Define the group 
\[ \bast\mathcal{O}(\uptheta)  = \{\mathcal{O}\text{-diophantine approximations of }\uptheta  \} \]
as well as the dual group $\bast\mathcal{O}(\uptheta)^{\perp}$, equal to $\bast\mathcal{O}(\uptheta^{-1})$ for $\uptheta\not=0$,
the group of errors $\bast {\K}_{\upvarepsilon}(\uptheta )$ and the error map $\upvarepsilon_{K}(\uptheta)$.  
When $K=\Q$ we recover the diophantine approximation groups defined in \S \ref{ratapprox}.
Every $\Z$-diophantine approximation defines canonically an $\mathcal{O}$-diophantine approximation, and
identifying $\bast\Z (\uptheta )$ with its image in $\bast\mathcal{O}(\uptheta )$  we obtain an 
extension of groups  
\[ \bast\Z (\uptheta )\subset \bast\mathcal{O}(\uptheta).\]  In general,
for $L/K$ an extension of algebraic number fields, we obtain the extension $ \bast\mathcal{O}_{K}(\uptheta)\subset \bast\mathcal{O}_{L}(\uptheta)$.

\begin{exam} Let $K=\Q (\sqrt{D})$ for $D>0$ a square free integer.  
If $\bast n\in\bast\Z (\uptheta )$ then 
\[ \bast\mathcal{O}(\uptheta )\supsetneq   \bast \Z(\uptheta )\oplus  \bast \Z(\uptheta )\cdot \sqrt{D}.\]
In fact, if $\bast m +\bast n\sqrt{D}\in \bast \Z(\uptheta )\oplus  \bast \Z(\uptheta )\cdot \sqrt{D}$ then 
$(\bast m +\bast n\sqrt{D})^{\perp} = \bast m^{\perp} +\bast n^{\perp}\sqrt{D}$.  
On the other hand, there exist $\bast m,\bast n\not\in\bast\Z (\uptheta )$
for which $\bast m+\bast n \sqrt{D}\in \bast\mathcal{O}(\uptheta )$.  For instance, let $\bast m,\bast n$ be such
that there exist $\bast m',\bast n'\in\bast\Z$ with
$ \bast m\uptheta-\bast m'\simeq 1/2$ and $ \bast n\sqrt{D}\uptheta-\bast n '\sqrt{D}\simeq -1/2$.
Then $\bast m + \bast n\sqrt{D}\in \bast\mathcal{O}(\uptheta )$
with $(\bast m + \bast n\sqrt{D})^{\perp} = \bast m'+\bast n'\sqrt{D}$.  This shows that we cannot
obtain $\bast\mathcal{O}(\uptheta )$ from $\bast \Z(\uptheta )$ simply by changing coefficients (tensor product), as we can in the case
of ideals in rings of integers.
\end{exam} 

The (ultrapowers of) principal ideals of $\mathcal{O}$ are recovered as follows.  For any element ${\boldsymbol z}\in {\K}$, define $\bast\mathcal{O}({\boldsymbol z})$ as the group
of $\mathcal{O}$-diophantine approximations of $\boldsymbol z$:
the subgroup of $\bast\upalpha\in \bast \mathcal{O}$ for which there exists $\bast\upalpha^{\perp}\in  \bast \mathcal{O}$ with 
$\bast\upalpha {\boldsymbol z}-\bast\upalpha^{\perp}$ infinitesimal.  We also write $\bast\Z (\boldsymbol z)$
for the set of $\bast n\in\bast\Z$ for which there exists $\bast n^{\perp}\in \bast\Z$ 
with $\bast n {\boldsymbol z}-\bast n^{\perp}$ infinitesimal: note that $\bast\Z (\boldsymbol z)\subsetneqq\bast\Z\cap\bast\mathcal{O}({\boldsymbol z})$.
Let $\upbeta\in K$ be non-integer 
and identify $\upbeta$ with $(\upbeta_{\upnu})\in\K$.
Then 
\[ \bast\mathcal{O}(\upbeta ) = \bast ( \upbeta)^{-1} \subset\bast\mathcal{O}  \]
where $ (\upbeta)^{-1}\subset\mathcal{O}$ is the inverse of the fractional ideal $(\upbeta )$.
Thus for $\upbeta\in K$, there are {\it two} types of diophantine approximation group available:

\begin{itemize}
\item[{\bf {\small (global)}}] $\bast\mathcal{O}(\upbeta )$.  Here the status of $\upbeta$ as an element of the abstract global field $K$ is preserved,
and $\upbeta$ is approximated by $\bast\mathcal{O}$, both of which have been embedded ``globally'' in $\K$ i.e. via the complete collection of place embeddings.
\item[{\bf {\small ($\boldsymbol\upnu$-local)}}] $\bast\mathcal{O}(\upbeta_{\upnu})$.  Here we view $\upbeta$ an element of $\R$ via a particular place $\upnu$ of $K$, however
the approximations by $\mathcal{O}$ continue to take place globally.  For all intents and purposes, $\upbeta$ is no longer in $K$ i.e.\ is not recognized as such by the global approximator $\mathcal{O}$.
\end{itemize}
 
Certain constructions native to (additive) algebraic number theory extend 
naturally to the setting of diophantine approximation groups, and we mention these now:

\begin{itemize}
\item The trace map ${\sf Tr}:K\rightarrow \Q$ induces a homomorphism of diophantine approximation groups: for $\uptheta\in\R$,
\[ {\sf Tr}: \bast \mathcal{O}(\uptheta)\rightarrow \bast\Z (\uptheta ),\quad \bast\upalpha\mapsto 
{\sf Tr}(\bast\upalpha).\]
\item If $K/\Q$ is Galois and $\upsigma\in {\rm Gal}(K/\Q)$ then $\upsigma$ acts on ${\K}$ permuting
coordinates according to its action on places.  In particular, we have that $\upsigma$ induces
an automorphism 
\[  \upsigma : \bast\mathcal{O}(\uptheta)\longrightarrow \bast\mathcal{O}(\uptheta)\]
acting identically on $\bast\Z(\uptheta)$.  If $\upbeta\in K$ then we have in addition
\[   \upsigma (\bast\mathcal{O}(\upbeta_{\upnu} ) )=\bast \mathcal{O}(\upbeta_{\upsigma\upnu}).\]
\item Every ideal 
$\mathfrak{a}\subset \mathcal{O}$ gives rise to a diophantine approximation group: \[ \bast\mathfrak{a}(\uptheta) := \bast\mathcal{O}(\uptheta)\cap \bast\mathfrak{a} \] and if
$\upbeta\in K$, $\bast\mathfrak{a} (\upbeta) := \mathcal{O}(\upbeta)\cap \bast\mathfrak{a}$.
These are permuted by the Galois group in the expected way e.g.\
$ \upsigma ( \bast\mathfrak{a}(\uptheta)) =  \bast\upsigma (\mathfrak{a})(\uptheta)$ and $\upsigma (\bast\mathfrak{a}(\upbeta_{\upnu})) = \bast\upsigma ( \mathfrak{a})(\upbeta_{\upsigma\upnu})$.
\end{itemize}

The sequence classes appearing in 
(\ref{ODA}) differ already in several respects from 
those studied in other accounts.
 In \cite{Bu}, \cite{Sch}, the approximants are allowed to be arbitrary elements of $\upalpha\in K$,
not just integers, and the measure of complexity of $\upalpha$ is no longer ``the denominator'' but rather the height $H_{K}(\upalpha )$ of the 
polynomial $P_{K,\upalpha}(X)=\prod (X-\upnu (\upalpha ))$ as $\upnu$ runs over the places of $K$.  Thus,
 the custom is to study solutions $\upalpha\in K$ to the inequality 
\[ |\uptheta -\upalpha|<c H(\upalpha )^{-n}\] for appropriate choices of $c$, $n$.  

The 
 definition appearing in (\ref{ODA}) restores in a sense the integral lattice perspective
 from which one may view Diophantine Approximation over the rationals:  
Since the notion of Diophantine Approximation by algebraic numbers studied here is not the usual one i.e.\ that of \cite{Bu}, \cite{Sch},
 it would be of interest to formulate and prove a simple
analogue of Dirichlet's Theorem in the this setting.  

Given an integral basis $\mathfrak{A}=\{\upalpha_{1},\dots ,\upalpha_{d}\}$
of $K\subset\K$, let $\|\cdot \|_{\mathfrak{A}}$ be the associated norm: thus if $\boldsymbol x = \sum a_{i}\upalpha_{i}$
then $\|\boldsymbol x\|_{\mathfrak{A}} = \sqrt{\sum a_{i}^{2}}$.
 For $\upeta= \sum n_{i}\upalpha_{i}\in\mathcal{O}$ satisfying $n^{-1}_{i}\not=0$ for all $i$ we define 
 \[ \upeta^{-1}_{\mathfrak{A}}:= \sum n_{i}^{-1}\upalpha_{i} . \]
 For $\mathfrak{A}$ = the canonical orthonormal basis of $\K$, $ \upeta^{-1}_{\mathfrak{A}}= \upeta^{-1}$ = the multiplicative inverse of $\upeta$ in $\K$.
 
 We will write 
 $ \upeta>_{\mathfrak{A}} 0$ resp.\ $ \upeta\geq_{\mathfrak{A}} 0$ 
 if $n_{i}>0$ resp.\ $n_{i}\geq 0$ for all $i$. We will also
 write $\upeta
 >_{\mathfrak{A}} \upgamma$ if $\upeta-\upgamma>_{\mathfrak{A}} 0$.  Notice in this event that 
$ \|\upeta\|_{\mathfrak{A}}>\|\upgamma\|_{\mathfrak{A}}$ and 
   $\|\upeta_{\mathfrak{A}}^{-1}\|_{\mathfrak{A}}<\|\upgamma_{\mathfrak{A}}^{-1}\|_{\mathfrak{A}}$,
   the latter provided the inverses are defined.
Finally, we say that $\uptheta$ is {\bf {\small $K$-irrational}} if $\uptheta\not\in\upnu (K)$ for every archimedean place $\upnu$.
The Theorem which follows reduces to the classical Dirichlet Theorem upon taking $K=\Q$, since $n^{-1}_{\mathfrak{A}}=n^{-1}$.

 \begin{theo}[$K$-Dirichlet Theorem]\label{vecDir} Let $K$ be totally real.  Let $\mathfrak{A}=\{\upalpha_{1},\dots ,\upalpha_{d}\}$ be an integral basis of $K$ and suppose that $\uptheta$ is a real $K$-irrationality.
Then
 for any $\upeta\in \mathcal{O}$ 
 with $\upeta>_{\mathfrak{A}} 0$, 
 there exists $\upgamma,\upgamma^{\perp}\in \mathcal{O}$ with  $\|\upgamma\|_{\mathfrak{A}}<\|\upeta\|_{\mathfrak{A}} $ and 
 \[  \|  \upgamma\uptheta -\upgamma^{\perp}  \|_{\mathfrak{A}}< \|\upeta_{\mathfrak{A}}^{-1}\|_{\mathfrak{A}}. \]
 \end{theo}
 
 \begin{proof}  Pigeonhole principal in many dimensions.  Write $\upeta = \sum n_{i}\upalpha_{i}$.   Denote by $P$ the (closed) parallelopiped spanned by the $\upalpha_{i}$.  Let $ [0,\upeta ) := \{ \upbeta\in\mathcal{O}|\; 0\leq_{\mathfrak{A}}\upbeta <_{\mathfrak{A}} \upeta\} $.
 Thus every $\upbeta\in [0,\upeta )$ is of the form $\upbeta=\sum a_{i}\upalpha_{i}$ with $a_{i}\in\Z$
 satisfying $0\leq a_{i}<n_{i}$.
Note that the cardinality of $[0,\upeta )$ is $N=\prod n_{i}$.  
 Subdivide $P$ into $N$ parallelopipeds, each of which is a translate of the parallelopiped
$P_{N} = {\rm span}( n_{1}^{-1}\upalpha_{1},\dots , n_{d}^{-1}\upalpha_{d})$.  
  For each 
 $\upbeta\in [0,\upeta )$ let $ \{\upbeta\uptheta\}\in P$ 
 be the unique point of $P$
 which is an $\mathcal{O}$-translate of $\upbeta\uptheta$
 (unique since $\uptheta$ is $K$-irrational): thus, there exists $\upbeta^{\perp}\in\mathcal{O}$
 with $\{\upbeta\uptheta\}=\upbeta\uptheta-\upbeta^{\perp}\in P$.
Let $1_{\mathfrak{A}}=\sum \upalpha_{i}\in\mathcal{O}\cap P$.
By the pigeon-hole principle, there must be, amongst the $N+1$ numbers
$  \big\{ \{\upbeta\uptheta\}|\; \upbeta\in [0,\upeta )\big\} \cup \{ 1_{\mathfrak{A}} \}$, 
two which belong to the same subparallelopiped of $P$.
It follows that there exist  $\upbeta_{1},\upbeta_{2}\in [0,\upeta )$, $\upbeta_{1}^{\perp},\upbeta_{2}^{\perp}\in\mathcal{O}$ for which 
$ (\upbeta_{1}\uptheta-\upbeta_{1}^{\perp})-(\upbeta_{2}\uptheta-\upbeta_{2}^{\perp}) \in P^{\upsigma}_{N}$
where $\upsigma = (\upsigma_{i})\in\{\pm\}^{d}$ and
$P^{\upsigma}_{N}  :=  {\rm span}( \upsigma_{1}n_{1}^{-1}\upalpha_{1},\dots , \upsigma_{d}n_{d}^{-1}\upalpha_{d})$.
Let $\upgamma=\upbeta_{1}-\upbeta_{2}$ and $\upgamma^{\perp}=\upbeta^{\perp}_{1}-\upbeta^{\perp}_{2}$.   Then $\upgamma=\sum a_{i}\upalpha_{i}$ with $|a_{i}|<n_{i}$, so that
 $\|\upgamma\|_{\mathfrak{A} }<\|\upeta \|_{\mathfrak{A}}$.  
 Since $ \upgamma\uptheta -\upgamma^{\perp} \in
 P_{N}^{\upsigma}$,
 $   \|  \upgamma\uptheta -\upgamma^{\perp}  \|_{\mathfrak{A}}< \sqrt{\sum n_{i}^{-2}} =  
 \| \upeta_{\mathfrak{A}}^{-1}\|_{\mathfrak{A}}$ .
 
 \end{proof}

 Although it is usually left unstated, the approximations by $\upalpha\in K$ 
 occurring in \cite{Bu}, \cite{Sch} are made with 
 regard to a fixed real place embedding $\upnu : K\rightarrow\R$, and are thus {\it local} in nature.  The definition (\ref{ODA}) 
 is {\it global}, in that it regards $\bast\upalpha$ as a vector
 in $\K$ i.e.\ an element of the lattice defined by $\mathcal{O}$ in $\K$, and as such, 
 is allied
to the notion of  {\it approximation by conjugates} 
\cite{RoWa}.  Indeed, condition (\ref{ODA}) is equivalent to asking that 
\begin{align}\label{conjapprox} 
\uptheta\bast\upalpha_{\upnu}-
\bast\upalpha^{\perp}_{\upnu}\in\bast \C_{\upvarepsilon}
\end{align} 
for every infinite place
\footnote{In \cite{RoWa}, 
a bound is provided for the number of places for which $\bast\upalpha$ satisfies (\ref{conjapprox})
with Wirsing-Davenport-Schmidt type bounds (\cite{Wi}
\cite{DS}) on the errors. It may be that this work is related to the problem of finding the optimal exponent $\upkappa>1$
for which
\[  \|  \upgamma\uptheta -\upgamma^{\perp}  \|_{\mathfrak{A}}<  \|\upeta^{-1}\|^{\upkappa}_{\mathfrak{A}} \]
has a solution for infinitely many $\upeta$, for a fixed basis $\mathfrak{A}$.}.
It is natural then to consider the $\upnu$-local group $\bast\mathcal{O}(\uptheta)_{\upnu}$ of 
$\bast\upalpha\in\bast \mathcal{O}$ for which there exists  $\bast\upalpha^{\perp}\in\bast \mathcal{O}$
such that $\bast\upalpha_{\upnu}\uptheta -\bast\upalpha^{\perp}_{\upnu}\in\bast \C_{\upvarepsilon}$
for a fixed archimedean place $\upnu$.   However if $\upnu (\mathcal{O})\cap\R$ is dense in $\R$ (e.g.\ if $\upnu$ is real) then 
\[   \bast\mathcal{O}(\uptheta )_{\upnu} = \bast\mathcal{O}(\uptheta )_{\upnu}^{\perp}= \bast \mathcal{O},\]
so nothing interesting is obtained.
Notwithstanding, the $\upnu$-local error map
\[ \upvarepsilon_{K}(\uptheta )_{\upnu}:\bast\mathcal{O}(\uptheta )_{\upnu}\longrightarrow 
\bast\C_{\upvarepsilon}(\uptheta )\] is non-trivial and gives an obvious test for $K$-rationality.  Note that we have $\upvarepsilon_{K}(\uptheta )=(\upvarepsilon_{K}(\uptheta )_{\upnu})$.

\begin{prop}\label{vlocKerErrMap} There exists $\upnu$ with $\uptheta\in \upnu(K)$ $\Leftrightarrow$ there exists $\upnu$ with ${\rm Ker}(\upvarepsilon_{K}(\uptheta )_{\upnu})\not=0$.
\end{prop}

\begin{proof}  $\uptheta\in \upnu (K)$ $\Leftrightarrow$ there exists $\upalpha\in \mathcal{O}$ such that $\upalpha_{\upnu}\uptheta\in\mathcal{O}$
$\Leftrightarrow$ ${\rm Ker}(\upvarepsilon_{K}(\uptheta )_{\upnu})\not=0$.
\end{proof}

The following Proposition shows that the global groups $ \bast\mathcal{O}(\uptheta) $ {\it do not} recognize when $\uptheta\in\upnu (K)$ 
as the groups $\bast\Z (\uptheta )$ recognize $\uptheta\in\Q$.   The problem (as has already been alluded to above) is that we are embedding $\uptheta$
diagonally as an element of $\R$, and not by the $K$-places.

\begin{prop}\label{globaldoesntdist}  Let $\uptheta\in \upnu(K)-\Q$ for some $\upnu$.  Then  \[ \bast\mathcal{O}(\uptheta)\cap \mathcal{O}=0.\]  In particular,
the global error map $\upvarepsilon_{K} (\uptheta):\bast\mathcal{O}(\uptheta)
\rightarrow \bast\mathbb{K}_{\upvarepsilon}(\uptheta )$ is an isomorphism.
\end{prop}

\begin{proof}  Suppose first that $K/\Q$ is Galois and that there exists 
$0\not=\upgamma\in \bast\mathcal{O}(\uptheta)\cap \mathcal{O}$.  Then since $\mathcal{O}\subset\K$ is discrete, we must have
$\upgamma^{\perp}\in \mathcal{O}$ and therefore $\upvarepsilon_{K}(\uptheta )_{\upnu}(\upgamma )=0$ for all $\upnu$.
Thus $\upgamma^{\perp} = \upgamma\uptheta= (\upgamma_{1}\uptheta,\dots , \upgamma_{d}\uptheta)$ in  $\K$,
and the coordinates of this vector form a complete set of conjugates for $\upgamma^{\perp}$ w.r.t.\ the action of ${\rm Gal}(K/\Q )$.  But this can only be true if $\uptheta\in\Q$, contrary to our hypotheses.  If $K/\Q$ is not Galois, let $L/\Q$ be Galois containing
$K$ with ring of integers $\mathcal{O}_{L}$.  If there exists $0\not=\upgamma\in \bast\mathcal{O}_{K}(\uptheta)\cap \mathcal{O}_{K}$ then $\upgamma\in \bast\mathcal{O}_{L}(\uptheta)\cap \mathcal{O}_{L}$, and the previous argument applies.
\end{proof}

We can obtain a non-trivial group of
 $\upnu$-local diophantine approximations using a variant of the group of denominator numerator pairs
 introduced at the end of \S \ref{ResClassGps}.
Define 
\[  \bast \mathcal{O}^{1,1}(\uptheta)\] to be the group of
pairs $(\bast \upalpha ,\bast \upalpha^{\perp})$ where $\bast\upalpha\in \bast\mathcal{O}(\uptheta)$,
and let $\bast\mathcal{O}^{1,1}(\uptheta)_{\upnu}$ be the associated $\upnu$-local version. 
Note that we get a proper subgroup 
\[ \bast \mathcal{O}^{1,1}(\uptheta)_{\upnu}\subsetneqq
 \bast \mathcal{O}^{2}\quad\text{ and }\quad
   \bast \mathcal{O}^{1,1}(\uptheta) = \bigcap_{\upnu} \bast\mathcal{O}^{1,1}(\uptheta)_{\upnu} .\]   
In this sense, we may think of the $\bast \mathcal{O}^{1,1}(\uptheta)_{\upnu}$ as the $\upnu$-local ``factors''
of $\bast \mathcal{O}^{1,1}(\uptheta)$.

 The approximations studied in \cite{Bu}, \cite{Sch} canonically produce elements of
the local groups $\bast \mathcal{O}^{1,1}(\uptheta)_{\upnu}$ for which the denominator belongs to $\bast\Z$.
For example, consider the main result in Chapter VIII, \S 2 of \cite{Sch}: for a {\it real} algebraic number field $K$ there exists
 $c_{K}$ such that for all $\uptheta\not\in K$, the inequality 
 \begin{align}\label{tradoapprox}    |\uptheta-\upalpha| <c_{K}\max \{1, |\uptheta|^{2}\}H_{K}(\upalpha )^{-2} \end{align}
 has infinitely many solutions $\upalpha\in K$.  
 
 To put this in the $\upnu$-local context, fix the embedding $\upnu :K\hookrightarrow\R$
implicit in the formulation of (\ref{tradoapprox}) and identify $\upalpha$ with $\upnu (\upalpha )$. 
 Let $Q_{K,\upalpha}(X)=a_{d}X^{d}+\cdots + a_{1}X + a_{0}$ be the polynomial proportional to $P_{K,\upalpha}(X)=\prod (X-\upnu (\upalpha ))$ with rational integral coprime coefficients, so that $H_{K}(\upalpha ) = \max |a_{i}|$, and set $a=a_{d}$.  Then $\tilde{\upalpha} :=a\upalpha\in\mathcal{O}_{K}$ (see \cite{Ne}, pg.\ 8).  Since $|a|/H_{K}(\upalpha )\leq 1$ it follows that if $\upalpha$ is a solution of (\ref{tradoapprox}) then
$  |a\uptheta-\tilde{\upalpha}| <c_{K}\max \{1, |\uptheta|^{2}\}H_{K}(\upalpha )^{-1}$.    
 Thus any infinite sequence $\{ \upalpha_{i}\}\subset K$ satisfying (\ref{tradoapprox}) gives rise to 
 \[  ( \bast \tilde{\upalpha} ,\bast a)  \in \bast \mathcal{O}^{1,1}(\uptheta)_{\upnu} \]
 where $\bast a \in\bast\Z$ is the sequence class of leading coefficients of the polynomials $Q_{K,\upalpha_{i}}(X)$.

Note that $ ( \bast \tilde{\upalpha} ,\bast a)$ need not define an element of
 the global group $\bast \mathcal{O}^{1,1}(\uptheta)$.  For example, suppose that $K=\Q (\sqrt{D})$ and $\bast\tilde{\upalpha} = \bast m + \bast n\sqrt{D}$.  Then
 $( \bast \tilde{\upalpha} ,\bast a)\in \bast \mathcal{O}^{1,1}(\uptheta)$ if and only if $\bast n\in\bast\Z (\sqrt{D})$.
To find a sequence $\{\upgamma_{i}\}\subset \mathcal{O}$
 obtained from $\{\upalpha_{i}\}$ satisfying (\ref{tradoapprox}), one would be better served by replacing $a_{i}$ by a sequence
 of ``denominators'' belonging strictly to $\mathcal{O}_{K}$.

We now define some foliated spaces. Given $\boldsymbol z\in\K$, let $V_{\mathbb{K}}(\boldsymbol z )\subset\mathbb{K}^{2}$ be the graph of the map $\boldsymbol w \mapsto \boldsymbol z\boldsymbol w$: the 
{\bf {\small $\mathbb{K}$-Kronecker foliation}} 
\[  \mathfrak{F}_{\mathbb{K}}(\boldsymbol z)\] of slope $\boldsymbol z$  is the foliation of the torus $\T^{2}_{K}=\mathbb{K}^{2}/\mathcal{O}^{2}$ 
obtained by projecting onto $\T^{2}_{K}$ the linear foliation of $\mathbb{K}^{2}$ by planes $\boldsymbol w + V_{\mathbb{K}}(\boldsymbol z)$.
For $\uptheta\in\R-\Q$ diagonally embedded in $\mathbb{K}$ we write $\mathfrak{F}_{\mathbb{K}}(\uptheta )$.  
Note that we have an inclusion $\mathfrak{F}(\uptheta )\hookrightarrow
\mathfrak{F}_{\mathbb{K}}(\uptheta )$ of the classical Kronecker foliation, induced by the diagonal
map $\R\hookrightarrow\mathbb{K}$.
The following Proposition is the geometric parallel of Proposition \ref{globaldoesntdist}:

\begin{prop} Let $\uptheta\in\R-\Q$. The leaves of $\mathfrak{F}_{\mathbb{K}}(\uptheta )$ are isomorphic to $\mathbb{K}$ and are dense,
and  $\mathfrak{F}_{\mathbb{K}}(\uptheta )$
fibers over $\T_{K}$.  There are canonical inclusions of $\mathbb{K}$ and $\T_{K}$ as the leaf resp. the fiber transversal through $\boldsymbol 0$.  
\end{prop}

\begin{proof} $\mathfrak{F}_{\mathbb{K}}(\uptheta )$ can be identified with the suspension
$(\mathbb{K}\times \T_{K}))/\mathcal{O}_{K}$ through the action
$(\boldsymbol w, \overline{\boldsymbol \upxi} )\mapsto (\boldsymbol w+\upalpha , \overline{\boldsymbol \upxi  - \uptheta\upalpha})$, and
since the action on the right factor by $\mathcal{O}_{K}$ has dense orbits, the inclusions $\mathbb{K}, \T_{K}\hookrightarrow
\mathbb{K}\times \T_{K}$ survive the quotient.
\end{proof}

  \begin{note} If $\upbeta\in K$, then $\mathfrak{F}_{\mathbb{K}}(\upbeta)$ ($\upbeta$ place-embedded in $\K$) is non-minimal having leaves
isomorphic to $\T_{(\upbeta )^{-1}} = \mathbb{K}/(\upbeta )^{-1}$.
\end{note}

 Let 
 \[ \bbull\mathbb{K} = \bast\mathbb{K}/\bast\mathbb{K}_{\upvarepsilon}=(\K\times\bast\mathcal{O} )/\mathcal{O}.\]  As in \S \ref{ResClassGps}, a section topology on $\bast\mathcal{O}$
 is furnished by an $\mathcal{O}$-set section $\upsigma$ to the standard part epimorpism $\bast\mathcal{O}\rightarrow\T_{K}$,
 used in turn to topologize $\bbull\K$ via the suspension structure
$\bbull\K =(\K\times\bast\mathcal{O} )/\mathcal{O}$. 
 The group ${\sf PSL}_{2}(\mathcal{O})$, acting on $\overline{\K}=\K\cup\{\infty\}$, provides the notion of {\bf {\small $K$-equivalence}}.
 For each element $\upalpha\in K\cup \{ \infty\}$ we may associate the transversal subgroup $\T_{K,\upalpha}$ with ``slope" $\upalpha$
 (as before $\T_{K,\infty}=\T_{K}$), obtaining
 new section topologies based on $\T_{K,\upalpha}$, and the set of such 
 ${\rm Top}_{K}(\uptheta )$ is natural with respect to $K$-equivalence.
 Writing \[ \hat{\mathbb{K}}_{\upalpha}=(\bbull \mathbb{K}, \mathbb{K},\bast\mathcal{O}_{\upalpha} ),\] equipped with a
 corresponding section topology, we have the analogue of Theorem \ref{kron}:
 \begin{prop}$   \hat{\mathbb{K}}_{\upalpha}/\bast\mathcal{O}(\uptheta)  \cong  (\mathfrak{F}_{\mathbb{K}}(\uptheta ), \T_{K,\upalpha} ) $.
 If $K/\Q$ is Galois then ${\rm Gal}(K/\Q)$  acts by
 automorphisms of the pair $ (\mathfrak{F}_{\mathbb{K}}(\uptheta ), \T_{K} )$ acting trivially 
 on  $(\mathfrak{F}(\uptheta ), \T )\subset (\mathfrak{F}_{\mathbb{K}}(\uptheta ), \T_{K} )$.

 \end{prop}
 \begin{proof} The proof is basically the same as Theorem \ref{kron}: the epimorphism $\bbull\K\rightarrow\T^{2}_{K}$ is defined
by mapping classes $\{\boldsymbol z_{i}\}\in\bbull\boldsymbol z$ to the leaf through $0$ of $\mathfrak{F}_{\K}(\uptheta)$
and the rest follows as in Theorem \ref{kron}.  If $K/\Q$ is Galois, then ${\rm Gal}(K/\Q)$ preserves the lattice $\bast\mathcal{O}(\uptheta)$
and the transversal $\T_{K}$, from which the last statement follows.
 \end{proof}
 
Let $\bbull\K^{2}$ have a product section topology coming from section topologies defined on the
extended lines $\bbull L_{K}(\uptheta )$ and $\bbull L_{K}(\uptheta^{-1} )$ (as in \S \ref{ResClassGps}) and define the extended Kronecker foliation 
\[ \mathfrak{F}^{1,1}_{\mathbb{K}} (\uptheta )=\bbull\mathbb{K}^{2}/\bast\mathcal{O}^{1,1}(\uptheta )\] following the prescription given in \S \ref{ResClassGps}, producing a foliation by planes isomorphic to $\mathbb{K}^{2}$. 
In this context we introduce foliations associated to the groups $\bast\mathcal{O}^{1,1}(\uptheta )_{\upnu}$,
\[  \mathfrak{F}^{1,1}_{\K}(\uptheta )_{\upnu} := \bbull\mathbb{K}^{2}/\bast\mathcal{O}^{1,1}(\uptheta )_{\upnu}. \]
Since $\mathfrak{F}^{1,1}_{\K}(\uptheta )_{\upnu}$ is a quotient of $\mathfrak{F}^{1,1}_{\K}(\uptheta )$, we give it the quotient
foliated structure, which may not be transversally Hausdorff.  Note that $\bast\mathcal{O}^{1,1}(\uptheta )_{\upnu}\not\subset  \bbull L_{\K}(\uptheta )$.
We regard $\mathfrak{F}^{1,1}_{\mathbb{K}} (\uptheta )$ as the ``compositum''
of the $\mathfrak{F}(\uptheta )^{1,1}_{\upnu}$: that is, $\mathfrak{F}^{1,1}_{\K}(\uptheta )$  
is the minimal foliation isogenously covering all of the $\mathfrak{F}^{1,1}_{\K}(\uptheta )_{\upnu}$ (the geometric factors of $\mathfrak{F}^{1,1}_{\mathbb{K}} (\uptheta )$).
The relationships between Kronecker foliations and their extended counterparts (including those defined in \S \ref{ResClassGps}) are summarized by the following diagram:
\begin{diagram}
& &\mathfrak{F}_{\K}(\uptheta ) & \subset &  \mathfrak{F}^{1,1}_{\K}(\uptheta ) &  &   \\ 
&\ruInto & & \ruInto& & & \\
\mathfrak{F}(\uptheta ) &\subset &\mathfrak{F}^{1,1}(\uptheta )	& & \dOnto_{{\rm pr}_{\upnu}} & &  \\
& & & \rdInto& & & \\
& & & &  \mathfrak{F}^{1,1}_{\K}(\uptheta )_{\upnu} & & 
\end{diagram}
where ${\rm pr}_{\upnu}$ is the natural projection and the diagonal arrows are inclusions.

\begin{prop} There exists $\upnu$ with $\uptheta\in \upnu(K)$ $\Leftrightarrow$ there exists $\upnu$ with $\bast\mathcal{O}^{1,1}(\uptheta )_{\upnu}\cap
\mathcal{O}^{2}\not=0$ $\Leftrightarrow$ there exists $\upnu$ such that $\mathfrak{F}^{1,1}(\uptheta )_{\upnu}$ has
nonplanar leaves.
\end{prop}

\begin{proof} The first bi-implication is clear.  If $\bast\mathcal{O}_{\upnu}^{1,1}(\uptheta )\cap
\mathcal{O}^{2}$ is non-trivial for some $\upnu$, the standard plane $\mathbb{K}^{2}\subset\bbull\K^{2}$ is stabilized by it
and its image in $\mathfrak{F}^{1,1}(\uptheta )_{\upnu}$ is non planar.  The converse is true  since
if one leaf has topology, all do, and the same topology since the action by $\bast\mathcal{O}^{1,1}(\uptheta )_{\upnu}$
is by translation.  This happens only if $\bast\mathcal{O}^{1,1}(\uptheta )_{\upnu}\cap
\mathcal{O}^{2}\not=0$. 
\end{proof}

Let ${\boldsymbol M}\in M_{r,s}({\K})$ =  the algebra of $r\times s$ matrices with entries in $\K$.  The diophantine approximation group $\bast\mathcal{O}^{s}({\boldsymbol M})$ is defined as in \S \ref{higher}.  For any algebraic extension $L/K$ we have the
diagonal place embeddings $\K\hookrightarrow \mathbb L$, and therefore 
$\bast \mathcal{O}_{K}^{s}({\boldsymbol M})\subset \bast \mathcal{O}_{L}^{s}({\boldsymbol M})$.
In the case of a matrix $\Uptheta\in M_{r,s}(\R)$ (diagonally embedded in $ M_{r,s}({\K})$) we write $\bast\mathcal{O}^{s}(\Uptheta )$, and we have $\bast\Z^{s}(\Uptheta )\subset \bast\mathcal{O}^{s}(\Uptheta )$.  The diophantine approximation groups 
$\bast\mathcal{O}^{s}(\Uptheta )$ are Galois natural.
The $\upnu$-local groups $\bast\mathcal{O}^{s}(\Uptheta)_{\upnu}$ may be defined as well; they are generically 
equal to $\bast\mathcal{O}^{s}$, but the groups  $\bast\mathcal{O}^{s,r}(\Uptheta)_{\upnu}\subset\bast\mathcal{O}^{s+r}$ of ``denominator numerator '' pairs
are non trivial.   Denote as before the
error map $\upvarepsilon_{K} (\Uptheta )_{\upnu}:\bast\mathcal{O}^{s}(\Uptheta)_{\upnu}\longrightarrow \bast\mathbb{R}_{\upvarepsilon}^{r}$.
Let $\bast\tilde{\mathcal{O}}^{s}(\Uptheta )_{\upnu}<
\bast\mathcal{O}^{s}(\Uptheta )_{\upnu}$ be the kernel of the duality map $\perp_{\upnu}: \bast\mathcal{O}^{s}(\Uptheta )_{\upnu}\longrightarrow \bast\mathcal{O}^{r}(\Uptheta )^{\perp}_{\upnu}$.  

It is in this context that we study (in)homogeneous independence over $K$, a concept which requires a word of explanation.  For each $i=1,\dots, r$, let 
\[ \boldsymbol e^{K}_{i}=(0,\dots ,0,1_{K},0,\dots 0)\in\K^{r},\] where $1_{K}=(1,\dots ,1)\in\K$ is in the $i$th place.  The set $\{\boldsymbol e^{K}_{i}\}$ forms a basis for $\K^{r}$ as a $\K$-module.  We say that a set  $\boldsymbol v_{1},\dots ,\boldsymbol v_{m}\in\K^{r}$ is (in)homogeneously independent over $K$
if the set $\boldsymbol v_{1},\dots ,\boldsymbol v_{m}$ (along with $\boldsymbol e^{K}_{1},\dots ,\boldsymbol e^{K}_{r}$) is independent over $K$.


\begin{theo}\label{KErrorTermHomDep} The columns of $\Uptheta$ are (inhomogeneously) independent over $K$ $\Leftrightarrow$ for all $\upnu$ the (in)homogeneous error map
\[ 
\left(\upvarepsilon_{K} (\Uptheta )_{\upnu}:\bast\mathcal{O}^{s}(\Uptheta )_{\upnu}\longrightarrow \bast\R^{r}_{\upvarepsilon}\right)\quad
\tilde{\upvarepsilon}_{K} (\Uptheta )_{\upnu}: \bast\tilde{\mathcal{O}}^{s}(\Uptheta )_{\upnu}\longrightarrow \bast\R^{r}_{\upvarepsilon}  \] 
is injective $\Leftrightarrow$ $(\bast\mathcal{O}^{s,r}(\Uptheta )_{\upnu})$ $\bast\tilde{\mathcal{O}}^{s,r}(\Uptheta )_{\upnu}$ contains no nontrivial $\bast\mathcal{O}^{r+s}$-ideals
for all $\upnu$.
\end{theo}

\begin{proof}  The columns of $\Uptheta$ (along with the basis $\boldsymbol e^{K}_{1},
\dots , \boldsymbol e^{K}_{r}$) are independent over $K$ $\Leftrightarrow$ no vector $\bast\boldsymbol\upalpha\in \bast\tilde{\mathcal{O}}^{s}(\Uptheta )$ (no vector $\bast\boldsymbol\upalpha\in \bast\mathcal{O}^{s}(\Uptheta ))$ has the property that $\Uptheta \bast\boldsymbol\upalpha_{\upnu}=0$ ($\Uptheta \bast\boldsymbol\upalpha_{\upnu}-\bast\boldsymbol\upalpha^{\perp}_{\upnu} =0$ for some $\bast\boldsymbol\upalpha^{\perp}\in\mathcal{O}$) $\Leftrightarrow$ the kernel of $\tilde{\upvarepsilon}_{\upnu} (\Uptheta )$ (of
$\upvarepsilon_{\upnu} (\Uptheta )$)
is trivial.  The ideal-theoretic statement is proved in the same way as its counterpart in Theorem \ref{homogeneouslinindep}. 
\end{proof}

We may define the $\K$-Kronecker foliation $\mathfrak{F}_{\K}(\Uptheta )$ of $\T_{K}^{r+s}$, using
the graph $\widetilde{L}_{\K}(\Uptheta )\subset\K^{s+r}$ of $\Uptheta:\K^{s}\rightarrow \K^{r}$.
Let 
$\hat{\mathbb{K}}^{s} = (\bbull\mathbb{K}^{s},\K^{s},\bast \mathcal{O}^{s})$, where $\bbull\mathbb{K}^{s}$
has been topologized using a section topology.

\begin{theo}  Let $\Uptheta$ be a real matrix whose columns are inhomogeneously independent over $\Q$. Then $\mathfrak{F}_{\K}(\Uptheta )$ has planar leaves and 
\[ \hat{\mathbb{K}}^{s}/\bast\mathcal{O} ^{s}(\Uptheta )\cong (\mathfrak{F}_{\mathbb{K}}(\Uptheta ), \T^{r}_{K}(\Uptheta ))
\]
where $\T^{r}_{K}(\Uptheta )$ is the image of the map $\bast\mathcal{O}\rightarrow \T^{r}_{K}$.  If in addition the rows of
$\Uptheta$ are homogeneously independent over $\Q$, then $\T^{r}_{K}(\Uptheta )=\T^{r}_{K}$.
\end{theo}

\begin{proof}  The condition of being inhomogeneously $\Q$-independent implies, by an argument similar to that
found in Proposition \ref{globaldoesntdist} of \S \ref{daringint}, that 
$\bast\mathcal{O} ^{s}(\Uptheta )\cap\mathcal{O}^{s}=0$.
If the rows of
$\Uptheta$ are homogeneously independent over $\Q$,
then the standard part map $\bast\mathcal{O}^{s}\rightarrow\T_{K}^{r}$
is an epimorphism. 
\end{proof}

Equivalent formulations for $K$-independence of vectors are made using extended local Kronecker foliations.
When $\Uptheta$ is invertible, we may topologize $\bbull\K^{s+r}$ using a direct sum of section topologies
along the transverse planes $\bbull L_{K}(\Uptheta )=\{(\bbull \boldsymbol z, \Uptheta \bbull \boldsymbol z)\}$,
$\bbull L_{K}(\Uptheta^{-1} )=\{(\bbull \boldsymbol z, \Uptheta^{-1} \bbull \boldsymbol z)\}$.  When
$\Uptheta$ is not invertible, we take a section topology on$\bbull L_{K}(\Uptheta )$ and on its complement
we simply put the discrete transverse topology.  In either event we define the $\upnu$-local factor  
\[ \mathfrak{F}^{s,r}_{\K}(\Uptheta )_{\upnu} := \bbull\K^{s+r}/ \bast\mathcal{O}^{s,r}(\Uptheta )_{\upnu},
\quad \] of the extended foliation
$\mathfrak{F}^{s,r}_{\K}(\Uptheta )$ as well as $\upnu$-local base foliated structure $\mathfrak{f}^{s,r}_{\K}(\Uptheta )_{\upnu}$.
The latter is defined as follows: consider the homogeneous subgroup $\bast\widetilde{\mathcal{O}}^{s,r}(\Uptheta )_{\upnu}<
\bast\mathcal{O}^{s,r}(\Uptheta )_{\upnu}< \bast\mathcal{O}^{r+s}$
of pairs of the form $(\bast\boldsymbol\upalpha, \boldsymbol 0)$: i.e.\ $\Uptheta \bast\boldsymbol\upalpha_{\upnu}\simeq \boldsymbol 0$. Note in fact that $\bast\widetilde{\mathcal{O}}^{s,r}(\Uptheta )_{\upnu}\subset \bbull\K^{s}$.  Then we define
\[  \mathfrak{f}^{s,r}_{\K}(\Uptheta)_{\upnu} := \bbull\K^{s}/\bast\widetilde{\mathcal{O}}^{s,r}(\Uptheta )_{\upnu}. \]

\begin{theo}\label{indeptheo2}   Let $\Uptheta$ be a real matrix. 
 \begin{itemize}
\item[i.] $\mathfrak{F}^{s,r}_{\K}(\Uptheta )_{\upnu}$ has planar leaves for all $\upnu$
$\Leftrightarrow$ 
 The columns of $\Uptheta$ are inhomogeneously independent over $K$.
\item[ii.] $\mathfrak{f}^{s,r}_{\K}(\Uptheta)_{\upnu}$ has planar leaves for all $\upnu$ $\Leftrightarrow$ 
The columns of $\Uptheta$ are homogeneously independent over $K$.
\end{itemize}
   \end{theo}

   \begin{proof} $\left\{ \Uptheta^{(1)},\dots , \Uptheta^{(s)}, {\boldsymbol e}^{K}_{1},\dots ,{\boldsymbol e}^{K}_{r}\right\}$ is $K$-independent $\Leftrightarrow$ the equation $\Uptheta\bast\boldsymbol\upalpha_{\upnu}-\bast\boldsymbol\upalpha^{\perp}_{\upnu}=\boldsymbol 0$
   has no solutions for all $\upnu$ $\Leftrightarrow$  $\K^{s+r}\cap \bast\mathcal{O}^{s,r}(\Uptheta )_{\upnu}=0$ for all $\upnu$
   $\Leftrightarrow$ the leaves of $\mathfrak{F}^{s,r}_{\K}(\Uptheta )_{\upnu}$ are planar for all $\upnu$.  An identical argument 
   applies to case ii.
   \end{proof}
   
   We many summarize the results of this section into a new dictionary entry:

\vspace{3mm}
\noindent\fbox{{\small {\bf $K$-Independence}}} \begin{center}
    \begin{tabular}{ | p{4cm} | p{2.5cm} | p{2.5cm} | p{2.5cm}|}
    \hline
    $\quad\quad\quad\quad\quad\;\;\;\Uptheta$ &  {\small error map}  & {\small DA group} & {\small Foliation} \\ \hline
    {\small inhomogeneously 
     $K$ independent columns}   & {\small $\forall\upnu,\; \upvarepsilon_{K} (\Uptheta )_{\upnu}$ is  injective}  & {\small $\forall\upnu,\; \bast\mathcal{O}^{s,r}(\Uptheta )_{\upnu}$ contains no non-0 $\bast\mathcal{O}^{s+r}$ ideals} & {\small 
     $\forall\upnu$, the leaves of $\;\mathfrak{F}^{s,r}_{\K}(\Uptheta )_{\upnu}$ are 1-connected}  \\ \hline
     {\small homogeneously 
     $K$ independent columns } &  {\small $\forall\upnu,\; \tilde{\upvarepsilon}_{K} (\Uptheta )_{\upnu}$ is  injective}  & {\small $\forall\upnu,\;\bast\tilde{\mathcal{O}}^{s,r}(\Uptheta )_{\upnu}$ contains no non-0 $\bast\mathcal{O}^{s+r}$ ideals} & {\small  $\forall\upnu$, the leaves of $\;\mathfrak{f}^{s,r}_{\K}(\Uptheta )_{\upnu}$ are 1-connected } \\ \hline
    \end{tabular}
\end{center}




\section{Polynomial Diophantine Approximation Rings}\label{polydio}

Let $\bast \Z[X]$ be the ring of polynomials with coefficients in $\bast\Z$ and let $\bast \Z[X]_{d}\subset\bast \Z[X]$ be the subgroup of polynomials of degree $\leq d$. 
Note that $\bast\Z[X] $ is a proper subring of the ultrapower $\bast (\Z [X])$, since
the latter has elements of unbounded degree.  A typical element $\bast f\in \bast\Z [X]$
can be represented as a class of sequence $\{ f_{i}\}$, where the degree of $f_{i}$
is equal to that of $\bast f$.
For $\uptheta\in\R$ we define the {\bf {\small ring of polynomial diophantine approximations}}
\[  \bast \Z[X](\uptheta ) = \big\{\bast f\in\bast\Z[X]\big|\;  \bast f(\uptheta)\simeq 0 \big\}:\]  
a ring since the defining infinitesimal equations can be added and multiplied.  Note that $ \bast \Z[X](\uptheta )$ is a ring without $1$
if $\uptheta\not=0$. 
If we denote  $\bast \Z[X]_{d}(\uptheta )=\bast \Z[X](\uptheta )\cap\bast \Z[X]_{d}$ then
\[ \bast \Z[X](\uptheta ) \cong \bigcup_{d\geq 1} \bast \Z[X]_{d}(\uptheta ).  \]
Thus $ \bast \Z[X](\uptheta )$ has the structure of a filtered ring.
The associated erro ring of is \[  \bast \R_{\upvarepsilon}^{\rm poly}(\uptheta )=\left\{ \bast f(\uptheta ) |\; \bast f(X)\in  \bast \Z[X](\uptheta )  \right\}
\subset\bast\R_{\upvarepsilon}.\] 
If we let \[ \bast\Z[X]^{\dag} \subset \bast\Z[X]\] be the subring of polynomials that have $0$ constant term,
then the image $ \bast \Z[X]^{\dag}(\uptheta )$ of  $\bast \Z[X](\uptheta )$ by the projection $\bast\Z[X]\rightarrow\bast\Z[X]^{\dag} $
is an additive subgroup of $ \bast \Z[X]^{\dag}$.
The following Proposition is the analogue of Propositions \ref{Qerrortermprop} and \ref{vlocKerErrMap}:

\begin{prop}\label{injprop}  The error homomorphism 
\[   \upvarepsilon^{\rm poly}(\uptheta ):\bast \Z[X](\uptheta ) \longrightarrow  \bast \R_{\upvarepsilon},
\quad \bast f\mapsto \bast f(\uptheta ) \]
 is injective 
$\Leftrightarrow$ $\uptheta$ is transcendental.
\end{prop}

\begin{proof}  If $\uptheta$ is transcendental then no non-zero polynomial $ \bast f\in\bast\Z [X]$
will have $\uptheta$ as a root: for if $\bast f (\uptheta )=0$ then there exists a representative sequence
$\{f_{i}\}$ in which $f_{i}(\uptheta )=0$ $\mathfrak{u}$-eventually.  Thus the image of
$\bast f$ by the evaluation map is $0$ only when $\bast f\equiv 0$.  On the other hand, if the evaluation map
is injective, than $\uptheta$ is not the root of any polynomial in $\bast\Z [X]$, hence not the root of any polynomial
$\Z[X]\subset \bast\Z [X]$,
so $\uptheta$ must be transcendental.
\end{proof}

Thus if $\uptheta$ is transcendental, the error map $ \upvarepsilon^{\rm poly}(\uptheta )$ is an isomorphism onto the error ring, which therefore inherits a filtration:
\[  \bast \R_{\upvarepsilon}^{\rm poly}(\uptheta )\cong  \bigcup_{d\geq 1} 
\bast \R_{\upvarepsilon}^{\rm poly_{d}} ,\quad \bast \R_{\upvarepsilon}^{\rm poly_{d}}:=  \bast\upvarepsilon^{\rm poly}(\uptheta )(\bast \Z[X]_{d}(\uptheta )). \]

The result which follows may be compared with Theorem \ref{groupideal}.

\begin{theo}\label{algebraic}  $\uptheta$ is algebraic $\Leftrightarrow$
$ \bast \Z[X](\uptheta )$ contains a non-zero ideal $I$ of $\bast \Z[X]$, and in the latter case, there is a maximal such 
ideal
\[ I(\uptheta) := \bast (m_{\uptheta} (X)) = {\rm Ker}(\bast \upvarepsilon^{\rm poly}(\uptheta )) \]
where $m_{\uptheta}(X)$ is the minimal polynomial of $\uptheta$.
\end{theo}

\begin{proof} If the group $ \bast \Z[X](\uptheta )$ contains an ideal $I$ of $\bast \Z[X]$, then $I$
consists of polynomials $\bast f$ for which $\uptheta$ is a zero, so that $\uptheta$ is
algebraic.  In this case, $I$ is contained in $I(\uptheta )$. 
\end{proof}

\begin{note} It is not difficult to see that $ \bast \Z[X](\uptheta )$ is an {\it ideal} of $\bast\Z [X]$ $\Leftrightarrow$ $\uptheta\in\Q$.
\end{note}

Theorem \ref{algebraic} makes plain that the polynomial diophantine approximation groups
are to the real algebraic numbers in much the same way that the rational diophantine approximation groups
$\bast\Z (\uptheta )$ are to $\Q$.  For this reason, in \cite{Ge5} we will see that the structure of the ``ideological arithmetic'' of the polynomial diophantine approximation groups collocates naturally with the Mahler classification.

Let us explore somewhat more explicitly the relationship between $ \bast \Z[X](\uptheta )$ and those diophantine approximation groups
examined in \S\S \ref{ratapprox}, \ref{higher}, \ref{daringint}. 
\begin{enumerate}
\item[\fbox{a}] We have
$ \bast\Z (\uptheta )\cong \bast\Z [X]_{1}(\uptheta )\cong \bast\Z [X]_{1}^{\dag}(\uptheta )$ via  $\bast n\mapsto \bast nX -\bast n^{\perp}\mapsto \bast nX$. 
\item[\fbox{b}]\label{vectorincl} Note that $\bast\Z [X]_{d}(\uptheta )$ is isomorphic
 to $\bast\Z^{d}(\Uptheta)$ for $\Uptheta = (\uptheta^{d},\dots ,\uptheta )$: the isomorphism is defined, for $\bast \boldsymbol n=(\bast n_{d},\dots ,\bast n_{1})\in \bast\Z^{d}(\Uptheta)$, by \[ \bast \boldsymbol n\mapsto \bast n_{d}X^{d} + \cdots + \bast n_{1}X - \bast \boldsymbol n^{\perp}.\]
 
 \item[\fbox{c}] Let $K/\Q$ be Galois of degree $d$ with ring of integers $\mathcal{O}$.  There is a {\it nonlinear} map
 \begin{align}\label{nonlinmap}
    \bast\mathcal{O} (\uptheta )\longrightarrow \bast\Z [X]_{d}(\uptheta ) 
    \end{align}
defined as follows.  
Given $\bast\upalpha\in \bast\mathcal{O} (\uptheta )$, denote by $\bast \upalpha_{\upnu}$
the conjugates of $\bast\upalpha$.  Then the polynomial
\begin{align}\label{polyasoc} \bast f(X)= \bast f_{\bast \upalpha}(X) = \prod_{\upnu} (\bast\upalpha_{\upnu} X - \bast\upalpha_{\upnu} ^{\perp}) 
\end{align}
belongs to $\bast\Z_{d} [X](\uptheta )$.  Indeed, $ \bast f(X) $ is invariant
with respect to the action of ${\rm Gal}(K/\Q)$ so that its coefficients belong to $\bast\Q$, and in fact to $\bast\Z$ since
$\bast \upalpha,\bast\upalpha^{\perp}\in\bast\mathcal{O}$.  Then \[ \bast f(\uptheta ) = \prod\upvarepsilon_{K} (\uptheta )_{\upnu}(\bast\upalpha )\simeq 0\] where $\upvarepsilon_{K} (\uptheta )_{\upnu}$ is
the $\upnu$-local error map.
The association $\bast\upalpha\mapsto \bast f_{\bast\upalpha}$ does not define a linear
map of $\bast\mathcal{O} (\uptheta )$ into $\bast\Z_{d} [X](\uptheta )$.  Rather, it is linear in each of the linear factors
of (\ref{polyasoc}):
if we write $ \bast f_{\bast \upalpha}(X)_{\upnu} = (\bast\upalpha_{\upnu} X - \bast\upalpha_{\upnu} ^{\perp})$ then
\[ \bast f_{\bast \upalpha + \bast \upbeta}(X)_{\upnu} =   \bast f_{\bast \upalpha}(X)_{\upnu} +
 f_{ \bast \upbeta}(X)_{\upnu} .\]  
 The map (\ref{nonlinmap}) may be viewed as giving an arithmetic parametrization of its image in $\bast\Z [X]_{d}(\uptheta)$ (effectively endowing it
 with the group structure of $ \bast\mathcal{O} (\uptheta )$).
 \end{enumerate}

The multi-variate version of polynomial diophantine approximations may be used to express
algebraic
independence.  Let $\boldsymbol \uptheta = (\uptheta_{1},\dots ,\uptheta_{s})$
be a set of $s$ real numbers.  Let $\boldsymbol X = (X_{1},\dots ,X_{s})$ and
denote by $\bast\Z[\boldsymbol X]$ the ring of polynomials in the multi-variable $\boldsymbol X$
with coefficients in $\bast\Z$. Define the sub ring
$ \bast\Z [\boldsymbol X ](\boldsymbol \uptheta )= \{ F(\boldsymbol X )|\;  
F(\boldsymbol \uptheta )\simeq 0 \} \subset \bast\Z[\boldsymbol X]$ and 
for $d\geq 0$ let $\bast\Z [\boldsymbol X ]_{d}(\boldsymbol \uptheta )$
to be the subgroup of polynomials of degree $\leq d$.

\begin{prop}\label{algind}  The coordinates of $\boldsymbol \uptheta = (\uptheta_{1},\dots ,\uptheta_{s})$ are algebraically independent
$\Leftrightarrow$ 
the error map $\upvarepsilon^{\rm poly} (\boldsymbol \uptheta ): \bast\Z [\boldsymbol X ](\boldsymbol \uptheta )
\rightarrow \bast\R_{\upvarepsilon}$ is injective.  
\end{prop}

\begin{proof}  Algebraic dependence of the coordinates of $\boldsymbol \uptheta$ just means that there
exists $F(\boldsymbol X )$ with $F(\boldsymbol \uptheta )=0$.  The proof then follows the argument of
Proposition \ref{injprop}.
\end{proof}

In addition, we have the analogue of Theorem \ref{algebraic}, whose proof is left to the reader.

\begin{theo}\label{algebraic2}  The coordinates of $\boldsymbol\uptheta$ are algebraically dependent $\Leftrightarrow$ 
$ \bast \Z[\boldsymbol X](\boldsymbol\uptheta )$ contains a non-zero ideal $I$ of $\bast \Z[\boldsymbol X]$, and in the latter case, there is a maximal such 
ideal $I(\boldsymbol\uptheta)  = {\rm Ker}( \upvarepsilon^{\rm poly}(\boldsymbol\uptheta ))$. 
\end{theo}

We now seek a geometric interpretation of the polynomial diophantine approximation
groups, that is, we wish to view $\bast\Z[X](\uptheta )$ as the fundamental germ 
of a Kronecker foliation.  In this connection, we will in fact want to work with
$\bast\Z [X]^{\dag}(\uptheta )$.  First, for each $n$ write $\boldsymbol\uptheta_{n}= (\uptheta,\dots ,\uptheta^{n})$ and let
\[ \mathfrak{F}^{\rm poly}(\uptheta ) := 
\underset{\longrightarrow}{\lim} \mathfrak{F}(\boldsymbol\uptheta_{n} )\] where the limit maps
are induced by the inclusions $\R^{n}\hookrightarrow \R^{n+1}$, $\boldsymbol r\mapsto (\boldsymbol r, 0)$.
Thus $ \mathfrak{F}^{\rm poly}(\uptheta )$ is a codimension 1 foliation of the direct limit torus $\T^{\infty}=\underset{\longrightarrow}{\lim}\T^{n}$
with infinite dimensional leaves.

Let $\bbull\R [X]_{n}$ be the vector space
of polynomials of degree $\leq n$ with coefficients in $\bbull\R$, and denote by $\bbull\R [X]_{n}^{\dag}$ the subspace
of polynomials with $0$ constant term, which is isomorphic to $\bbull\R^{n}$.  We may choose section topologies
on the $\bbull\R^{n}$ coming from the $\mathfrak{F}(\boldsymbol\uptheta_{n} )$ which are compatible with the inclusions $\bbull\boldsymbol r\mapsto (\bbull\boldsymbol r, 0)$
and give to $\bbull \R [X]^{\dag}$, by way of direct limit, a topology with respect to which we have \[ \bbull \R [X]^{\dag}/\bast\Z[X]^{\dag}(\uptheta )\cong
 \mathfrak{F}^{\rm poly}(\uptheta ).\]  The quotient $\bbull\R[X]/\bast\Z[X](\uptheta )$, suitably topologized, is isomorphic to the direct
 limit of extended foliations $ \mathfrak{F}^{n,1}(\boldsymbol\uptheta_{n})$.

\begin{theo} $ \mathfrak{F}^{\rm poly}(\uptheta ) $ has planar leaves $\Leftrightarrow$ 
$\uptheta$ is transcendental.
\end{theo}

\begin{proof}  By the isomorphism $ \mathfrak{F}^{\rm poly}(\uptheta )  \cong \bbull\R[X]^{\dag}/\bast\Z[X]^{\dag}(\uptheta )$, the leaf through $0$ is the image of $\R[X]^{\dag}$, which is stabilized precisely
by the elements of $\bast\Z[X]^{\dag}(\uptheta )\cap\Z [X]$.  The latter is non-trivial $\Leftrightarrow$ $\bast\Z[X](\uptheta )\cap\Z [X]$
is non-trivial $\Leftrightarrow$
$\uptheta$ is algebraic.
\end{proof}

Consider now a real vector $\boldsymbol \uptheta=(\uptheta_{1},\dots ,\uptheta_{s})$, and dictionary order the monomials
$\boldsymbol\uptheta^{ \bar{\imath}} :=\uptheta_{1}^{i_{1}}\cdots \uptheta_{s}^{i_{s}}$.  For each $d$ let 
$\boldsymbol\uptheta_{d}$ be the vector whose entries are the monomials $\boldsymbol\uptheta^{ \bar{\imath}} $
with $0<{\rm deg}(\bar{\imath}) = i_{1}+\cdots + i_{s}\leq d$, written from left to write according to the
dictionary order, a vector of size $N(d)$.
 In particular, $\boldsymbol\uptheta_{1}=\boldsymbol\uptheta$.  Then 
the coordinate inclusion defined by the dictionary ordering, $\bbull\R^{N(d)}\hookrightarrow 
\bbull\R^{N(d+1)}$, induces a monomorphism of Kronecker foliations
$   \mathfrak{F}(\boldsymbol\uptheta_{d})\hookrightarrow \mathfrak{F} (\boldsymbol\uptheta_{d+1}) $.
 We define
\[ \mathfrak{F}^{\rm poly}(\boldsymbol \uptheta )  := \underset{\longrightarrow}{\lim}\mathfrak{F}(\boldsymbol\uptheta_{d})
.\]  
Since $ \mathfrak{F}(\boldsymbol\uptheta_{d}) \cong \bbull\R [\boldsymbol X]^{\dag}_{ d}/\bast\Z [\boldsymbol X]^{\dag}_{ d}(\boldsymbol\uptheta )$ we have as in the case of $s=1$ that
\[ \mathfrak{F}^{\rm poly}(\boldsymbol \uptheta )  \cong \bbull\R[\boldsymbol X]^{\dag}/\bast\Z[\boldsymbol X]^{\dag}(\boldsymbol \uptheta ).\]
The proof of the following is left to the reader:
\begin{theo}\label{trivleavespoly} $\mathfrak{F}^{\rm poly}(\boldsymbol \uptheta )$ has planar leaves $\Leftrightarrow$ 
the coordinates of $\boldsymbol \uptheta$ are algebraically independent.
\end{theo}


\section{Rigidity in Transcendental Number Theory}\label{rigidity}

Many of the central  theorems and conjectures of Transcendental Number Theory ({\it e.g.}\ see Chapter 1 of \cite{Wa}) may be reformulated as
{\it rigidity statements}.   By this we mean the following: suppose that $X$ is a set equipped
with two $n$-ary relations \[ {\sf R}(x_{1},\dots ,x_{n}),\quad {\sf S}(x_{1},\dots ,x_{n})\quad {\text with}\quad {\sf R}\subset {\sf S}.\]  Given a subset
$X_{0}\subset X$, a {\bf {\small rigidity}} (rel $X_{0}$) is a statement which asserts that 
\[ {\sf R}={\sf S} \quad \text{ in } X_{0}^{n}.\]  
If we replace $X_{0}^{n}$ in the above by an arbitrary relation ${\sf X}\subset X^{n}$, we obtain the more general notion of a {\bf {\small relational rigidity}} (rel ${\sf X}$).
Examples of
rigidities include the Mostow rigidity theorem and the verified Poincar\'{e} conjecture.










We will also consider a more general notion of rigidity relative to a fixed function 
\[ f:X \rightarrow X.\]
Write $\boldsymbol x = (x_{1},\dots , x_{n})$ and $f(\boldsymbol x ) = (f(x_{1}),\dots , f(x_{n}))$.
Let ${\sf R}\subset {\sf S}$,  $X_{0}\subset X$ be as above. 
For each coordinate projection $\uppi :X^{2n}\rightarrow X^{n}$ let 
\[  {\sf S}_{\uppi}(f)=\{\boldsymbol x|\;  \uppi (\boldsymbol x, f(\boldsymbol x))\in {\sf S}\}.\] 
Note that when  $\uppi=\uppi_{\rm dom}$ = the domain projection then ${\sf S}_{\uppi}(f)={\sf S}$ and
when $\uppi=\uppi_{\rm ran}$ = the range projection then ${\sf S}_{\uppi}(f)=f^{\ast}{\sf S}$ = the pullback of ${\sf S}$.
We call \[  {\rm Gr}(f)^{\ast}{\sf S}:=\bigcap {\sf S}_{\uppi}(f)\] the {\bf {\small graph pullback}} of ${\sf S}$; clearly
${\sf S}\;\supset \;{\rm Gr}(f)^{\ast}{\sf S}\;\subset\; f^{\ast}{\sf S}$. 
 A {\bf {\small graph rigidity}} (rel $X_{0}$) is a statement which says that in $X_{0}$
 \[ {\rm Gr}(f)^{\ast}{\sf S}\subset {\sf R}. \]
 Every rigidity is a graph rigidity with respect to $f={\rm id}$: for in $X_{0}$, $ {\rm Gr}({\rm id})^{\ast}{\sf S}\subset {\sf S}={\sf R}$.
 The converse is not true: the Schanuel conjecture will be shown below to be a graph rigidity which is not a rigidity.
 
Under certain conditions, a graph rigidity will imply a family of (relational) rigidities corresponding to its coordinate
projections. For each coordinate projection 
\[ \uppi  (\boldsymbol x, f(\boldsymbol x)) =  (\boldsymbol x_{I},f(\boldsymbol x_{J})) := (x_{i_{1}},\dots ,x_{i_{k}},f(x_{j_{1}}),
\dots ,f(x_{j_{n-k}})) , \] let 
\[  \check{\uppi} (\boldsymbol x, f(\boldsymbol x)):=(\boldsymbol x_{\check{I}}, f(\boldsymbol x_{\check{J}}))\]
where $\check{I}$ resp.\ $\check{J}$ are the complements in $\{ 1,\dots ,n\}$ of $I$ resp.\ $J$.
We say that a graph rigidity is {\bf {\small strong}} (rel $Y_{0}$) if there exists $Y_{0}\subset X_{0}$ so that
for all $\uppi$, ${\sf S}_{\uppi}(f)={\sf R}$ (rel ${\sf Y}_{\uppi}(f) )$ where
\[ {\sf Y}_{\uppi} (f):= \{ \boldsymbol x |\; \check{\uppi} (\boldsymbol x, f(\boldsymbol x))\in Y^{n}_{0}  \}.\] 

Note that for $\uppi=\uppi_{\rm dom}$, $ {\sf Y}_{\uppi}(f)=(f^{-1}(Y_{0}))^{n}$; and for $\uppi = \uppi_{\rm ran}$, $ {\sf Y}_{\uppi}(f) =(Y_{0})^{n}$, so in these cases ${\sf S}_{\uppi}(f)={\sf R}$ are proper rigidities,
 the {\bf {\small domain}} and {\bf {\small range rigidities}}.  For arbitrary
 $\uppi$ we obtain relational rigidities rel ${\sf Y}_{\uppi}$.
 We will see that the Schanuel conjecture is strong but that an arbitrary graph rigidity need not be:
 the graph rigidity induced by a rigidity need only satisfy ${\sf S}_{\uppi}(f)={\sf R}$ rel ${\sf Y}_{\uppi}$ for $\uppi=\uppi_{\rm dom}$ or
$ \uppi_{\rm ran}$.




Consider the following $n$-ary relations in $\C$: 
\begin{itemize}
\item ${\sf LD}_{n}^{K}$ = the relation of a set of $n$ complex numbers being linearly dependent over $K$
an algebraic number field.
\item ${\sf  AD}_{n}$ = the relation of a set of $n$ complex numbers being algebraically dependent. 
\end{itemize}
For any extension $M/K$ of algebraic number fields, we have in $\C$ the inclusions 
\[ {\sf LD}_{n}^{K}\subset {\sf LD}_{n}^{M}\subset {\sf  AD}_{n},\] each of which is strict.

Let $\mathcal{L}$ be the $\Q$-vector space of logarithms of algebraic numbers: 
\[ \mathcal{L} = \big\{ \uplambda |\; e^{\uplambda }\in\bar{\Q}\big\}.\]  The homogeneous version of {\it Baker's Theorem} says that any $\Q$-linearly independent subset of $\mathcal{L}$ is $\bar{\Q}$-linearly independent; as a rigidity it 
may be re-stated:

\begin{theob} In $\mathcal{L}$,  ${\sf LD}_{n}^{\Q} ={\sf LD}_{n}^{\bar{\Q}}$ for all $n\geq 2$.
\end{theob} 

An important open problem in Transcendental
Number Theory is the {\it Conjecture on the Algebraic Independence of Logarithms}: any $\Q$-linearly independent subset of $\mathcal{L}$ is algebraically independent.
Its rigidity restatement is then:

\begin{logconj} In $\mathcal{L}$,  ${\sf LD}_{n}^{\Q} ={\sf AD}^{n}$ for all $n\geq 2$.
\end{logconj}
This Conjecture clearly contains Baker's Theorem as a subcase,
since ${\sf LD}_{n}^{\Q} \subset {\sf LD}_{n}^{\bar{\Q}}\subset {\sf AD}^{n}$.

There is another class of statements -- modeled on the Theorem of Lindemann-Weierstra\ss\ -- which
are (graph) rigidities expressed w.r.t.\ the exponential function $\exp:\C\rightarrow\C$.
Let $\boldsymbol \uptheta = (\uptheta_{1},\dots ,\uptheta_{n})$ 
be algebraic numbers which are linearly independent over $\Q$: the {\it Lindemann-Weierstra\ss\ Theorem} asserts that $\exp (\boldsymbol \uptheta)$ is
algebraically independent.  The rigidity formulation of this theorem is then

\begin{lwtheo}  In $\bar{\Q}$, $ \exp^{\ast}{\sf AD}^{n}= {\sf LD}_{n}^{\Q}$ for all $n\geq 2$.
\end{lwtheo}

Note that this is a rigidity statement since (in $\C$), ${\sf LD}_{n}^{\Q}\subset \exp^{\ast}{\sf AD}^{n}$.
Indeed, if $\boldsymbol\uptheta$ has linearly dependent coordinates, say 
\[ c_{i_{1}}\uptheta_{i_{1}} +\cdots +  c_{i_{k}}\uptheta_{i_{k}}  = d_{j_{1}}\uptheta_{j_{1}} + \cdots + d_{j_{l}}\uptheta_{j_{l}}\]
where $\{ i_{1},\dots , i_{k}; j_{1},\dots ,j_{l}\}\subset \{ 1,\dots ,n \}$ are distinct indices and the $c$'s and $d$'s are positive integers, then the
coordinates of $\exp (\boldsymbol\uptheta )$ satisfy the equation 
\[ X_{i_{1}}^{c_{i_{1}}}\cdots X_{i_{k}}^{c_{i_{k}}}= X_{j_{1}}^{d_{j_{1}}}\cdots X_{j_{l}}^{d_{j_{l}}}.\]  

 Let $\boldsymbol \uptheta\in \C^{n}$.
Recall that the {\it Schanuel Conjecture}
asserts that if the coordinates of $\boldsymbol\uptheta$ are linearly independent over $\Q$ then the transcendence degree of the extension $\Q (\boldsymbol\uptheta , \exp (\boldsymbol\uptheta))/\Q$ is at least $n$.   In other words, we have the following graph rigidity formulation
of the Schanuel conjecture:
 
\begin{schconj}In $\C$,  ${\rm Gr}(\exp )^{\ast}{\sf AD}^{n}\subset {\sf LD}_{n}^{\Q}$ for all $n\geq 2$. 
\end{schconj}

\begin{theo}  The Schanuel conjecture is a strong graph rigidity rel $\bar{\Q}$.  
\end{theo}

\begin{proof}  We must show that for each $\uppi$, 
\[ ({\sf AD}_{n})_{\uppi}(\exp) =  {\sf LD}_{n}^{\Q}  \quad \text{in }{\sf Y}_{\uppi}(\exp ). \] 
Write $\uppi(\boldsymbol z, \exp (\boldsymbol z))=(\boldsymbol z_{I},\exp (\boldsymbol z_{J}))$.  Notice that
${\sf Y}_{\uppi}(\exp )=\emptyset$ if $I\cap J\not=\emptyset$, by the Hermite-Lindemann Theorem, so the rigidity
is vacuously true for such $\uppi$.  We may then assume that $J=\check{I}=\{ 1,\dots ,n\}-I$.
 First we show that 
\begin{align}\label{inclusionldad}{\sf LD}_{n}^{\Q}\subset ({\sf AD}_{n})_{\uppi}(\exp) \quad \text{in } {\sf Y}_{\uppi}(\exp ).
\end{align}  Suppose that
 $\boldsymbol z = (z_{1},\dots ,z_{n})$
satisfies $\sum q_{i}z_{i}=0$ for $q_{i}\in\Z$ not all $0$. 
If $I=\emptyset$ then $\uppi=\uppi_{\rm ran}$ and $({\sf AD}_{n})_{\uppi}(\exp)=\exp^{\ast}{\sf AD}_{n}$, so (\ref{inclusionldad}) is trivial.
Thus we may write $I=\{ 1,\dots ,k\}$.  Either $I$ or $\check{I}$ contains an index for which $q_{i}\not=0$.
Suppose it is $I$.  Then we have $q_{1}z_{1}+\cdots +q_{k}z_{k}\in\bar{\Q}$ since $z_{k+1},\dots ,z_{n}\in\bar{\Q}$.
But this means that $(\boldsymbol z_{I},\exp(\boldsymbol z_{\check{I}}))\in {\sf AD}_{n}$, that is, $\boldsymbol z\in ({\sf AD}_{n})_{\uppi}(\exp)$.
On the other hand, if $J$ contains an index for which $q_{i}\not=0$, then we run a similar argument
for the exponentials: we see that
$ \exp (q_{k+1}z_{k+1})\cdots \exp (q_{n}z_{n})\in \bar{\Q}$, yielding a nontrivial algebraic relation
amongst the coordinates of $\exp (\boldsymbol z_{\check{I}})$.
We now demonstrate ${\sf LD}_{n}^{\Q}\supset ({\sf AD}_{n})_{\uppi}(\exp) $  in ${\sf Y}_{\uppi}(\exp )$.
Thus let $\boldsymbol z\in ({\sf AD}_{n})_{\uppi}(\exp) \cap {\sf Y}_{\uppi}(\exp )$: it is enough to show that
$\boldsymbol z\in {\rm Gr}(\exp )^{\ast}{\sf AD}^{n}$, since the Schanuel conjecture is already a graph rigidity.  Let $\uppi'\not=\uppi$.  Then we must have either $I'\cap \check{I}\not=\emptyset$
or $J'\cap \check{J}\not=\emptyset$.  In particular, it follows that some coordinate of $\uppi' (\boldsymbol z, \exp (\boldsymbol z))$
is algebraic (as we are assuming $\boldsymbol z\in {\sf Y}_{\uppi}(\exp )$), which implies $\boldsymbol z\in  ({\sf AD}_{n})_{\uppi'}(\exp)$.  It follows that $\boldsymbol z\in {\rm Gr}(\exp )^{\ast}{\sf AD}^{n}$ and we are done.
\end{proof}

\begin{note}  The domain and range rigidities of the Schanuel conjecture are the Logarithm Conjecture and the Lindemann-Weierstra\ss\
Theorem.
\end{note}



We conclude this section by translating the (graph) rigidity formulations of the {\it real variants} of the above statements into the language of diophantine approximation groups
and generalized Kronecker foliations.
To do this, it will be useful to summarize, for a row vector $\boldsymbol \uptheta = (\uptheta_{1},\dots ,\uptheta_{n})$ of real numbers, characterizations
of  dependence that have been proved thus far:
\begin{itemize}
\item[\fbox{1}] algebraically dependent
$\Leftrightarrow$ the polynomial error map $ \upvarepsilon^{\rm poly}( \boldsymbol \uptheta ): \bast\Z [\boldsymbol X ](\boldsymbol \uptheta )
\rightarrow \bast\R_{\upvarepsilon}$ has nontrivial kernel (Proposition \ref{algind}) $\Leftrightarrow$ 
$\bast\Z[\boldsymbol X](\boldsymbol\uptheta ]$ contains a nontrivial $\bast\Z[\boldsymbol X]$ ideal (Theorem \ref{algebraic2}) $\Leftrightarrow$ 
$\mathfrak{F}^{\rm poly}(\boldsymbol \uptheta )$ has nonplanar leaves (Theorem \ref{trivleavespoly}).
\item[\fbox{2}]  linearly dependent over $\Q$ $\Leftrightarrow$ the homogeneous error map 
$ \tilde{\upvarepsilon} (\boldsymbol \uptheta): \bast\widetilde{\Z}^{n}(\boldsymbol \uptheta )
\rightarrow \bast\R_{\upvarepsilon}$ has nontrivial kernel (Theorem \ref{homogeneouslinindep})
$\Leftrightarrow$ $\bast\widetilde{\Z}^{n}(\boldsymbol\uptheta )$ contains a nontrivial $\bast\Z^{n}$ ideal (Theorem \ref{homogeneouslinindep}) $\Leftrightarrow$ the homogeneous
Kronecker foliation $\mathfrak{f}(\boldsymbol\uptheta )$ has nonplanar leaves (Theorem \ref{indeptheohomog}).
\item[\fbox{3}]  linearly dependent over $K $ $\Leftrightarrow$ for some place $\upnu$ the homogeneous  error map 
$\tilde{\upvarepsilon}_{K}(\boldsymbol \uptheta)_{\upnu}: \bast\widetilde{\mathcal{O}}(\boldsymbol \uptheta )_{\upnu}
\rightarrow \bast\R_{\upvarepsilon} $ has non-trivial kernel (Theorem \ref{KErrorTermHomDep})
$\Leftrightarrow$ for some place $\upnu$ $\bast\widetilde{\mathcal{O}}^{n,1}(\boldsymbol \uptheta )_{\upnu}$ contains
a nontrivial $\bast\mathcal{O}^{n+1}$ ideal (Theorem \ref{KErrorTermHomDep}) $\Leftrightarrow$ for some place $\upnu$ the homogeneous
extended Kronecker foliation $\mathfrak{f}^{n,1}_{\K}(\boldsymbol\uptheta )_{\upnu}$ has nonplanar leaves (Theorem \ref{indeptheo2}).
\end{itemize}

Denote by $\mathcal{L}_{\R} = \mathcal{L}\cap\R$.
From items {\small \fbox{2}}, {\small \fbox{3}} and the rigidity formulation of Baker's Theorem, we have the following translations of the real Baker's Theorem:

\vspace{3mm}

\noindent\fbox{{\small {\bf $\R$ Baker's Theorem}}}   For all $\boldsymbol \uptheta\in \mathcal{L}^{n}\cap\R^{n}$

\vspace{2mm}

   \noindent \begin{tabular}{  | p{4cm} | p{4cm} | p{4cm} |}
    \hline
    {\small error maps} & {\small DA groups} & 
    {\small Kronecker foliations} \\ \hline
     {\small  $\exists \;K/\Q$, $\upnu$ with  $ {\rm Ker}\big( \tilde{\upvarepsilon}_{K} (\boldsymbol \uptheta)_{\upnu}\big)\not=0$ $\Rightarrow$
 ${\rm Ker}( \tilde{\upvarepsilon} (\boldsymbol \uptheta))\not=0$}  & {\small $\exists \;K/\Q$, $\upnu$ s.t. $\bast\widetilde{\mathcal{O}}^{n,1}(\boldsymbol \uptheta )_{\upnu}$ contains
a nontrivial $\bast\mathcal{O}^{n+1}$ ideal
$\Rightarrow$ $\bast\widetilde{\Z}^{n}(\boldsymbol\uptheta )$ contains a non-0 $\bast\Z^{n}$-ideal}  &  {\small $\exists \;K/\Q$, $\upnu$ s.t. 
$\mathfrak{f}^{n,1}_{\K}(\boldsymbol \uptheta)_{\upnu}$ has non 1-connected
leaves $\Rightarrow$ $\mathfrak{f}(\boldsymbol \uptheta)$
has non 1-connected
leaves}  \\ \hline
    \end{tabular}

\vspace{3mm}



From items {\small \fbox{1}} and {\small \fbox{2}} we have the translation of the real Logarithm conjecture stated in the introduction,
as well as the real versions of  Lindemann-Weierstra\ss\ Theorem and the Schanuel conjecture :


\newpage

\noindent\fbox{{\small {\bf $\R$ Lindemann-Weierstra\ss\ Theorem}}}   For all $\boldsymbol \uptheta\in\bar{\Q}^{n}\cap\R^{n}$:

\vspace{2mm}

   \noindent \begin{tabular}{  | p{4cm} | p{4cm} | p{4cm} |}
    \hline
    {\small error maps} & {\small DA groups} & 
    {\small Kronecker foliations} \\ \hline
     {\small  ${\rm Ker}\left( \upvarepsilon^{\rm poly}(\exp (\boldsymbol \uptheta )) \right)\not=0$ $\Rightarrow$
 ${\rm Ker}\left(\tilde{\upvarepsilon} (\boldsymbol \uptheta ) \right)\not=0$}  & {\small $\bast\Z[\boldsymbol X](\exp (\boldsymbol\uptheta ) )$ contains
a nontrivial $\bast\Z[\boldsymbol X]$ ideal
$\Rightarrow$ $\bast\widetilde{\Z}^{n}(\boldsymbol\uptheta )$ contains a non-0 $\bast\Z^{n}$-ideal}  &  {\small $\mathfrak{F}^{\rm poly}(\exp (\boldsymbol \uptheta) ) $ has non 1-connected
leaves $\Rightarrow$ $\mathfrak{f}(\boldsymbol \uptheta)$
has non 1-connected
leaves}  \\ \hline
    \end{tabular}

\vspace{3mm}

For any coordinate projection $\uppi:\R^{2n}\rightarrow\R^{n}$ denote by $\exp_{\uppi}(\boldsymbol\uptheta ) = \uppi \big(\boldsymbol \uptheta,\exp (\boldsymbol \uptheta)\big)$. 

\vspace{3mm}

\noindent\fbox{{\small {\bf $\R$ Schanuel Conjecture}}}   For all $\boldsymbol \uptheta\in\R^{n}$:

\vspace{2mm}

   \noindent \begin{tabular}{  | p{4cm} | p{4cm} | p{4cm} |}
    \hline
    {\small error maps} & {\small DA groups} & 
    {\small Kronecker foliations} \\ \hline
     {\small  $\forall\uppi$, $ {\rm Ker}\left( \upvarepsilon^{\rm poly}\big(\exp_{\uppi}(\boldsymbol\uptheta )\big) \right)\not=0$ $\Rightarrow$
 ${\rm Ker}( \tilde{\upvarepsilon} (\boldsymbol \uptheta))\not=0$}  & {\small $\forall\uppi$, 
 $\bast\Z[\boldsymbol X](\exp_{\uppi} (\boldsymbol\uptheta ) )$ contains
a nontrivial $\bast\Z[\boldsymbol X]$ ideal
$\Rightarrow$ $\bast\widetilde{\Z}^{n}(\boldsymbol\uptheta )$ contains a non-0 $\bast\Z^{n}$-ideal}  &  {\small $\forall\uppi$, $\mathfrak{F}^{\rm poly}(\exp_{\uppi}(\boldsymbol\uptheta ) )\subset \mathfrak{F}^{\rm poly}((\boldsymbol \uptheta,\exp (\boldsymbol \uptheta) ) $ has non 1-connected
leaves $\Rightarrow$ $\mathfrak{f}(\boldsymbol \uptheta)$
has non 1-connected
leaves}  \\ \hline
    \end{tabular}

\vspace{3mm}

We end with by sketching how the classical proof of Lindemann Weierstra\ss\ translates into the foliation-theoretic language.
Supposing the Theorem is false, the classical proof produces what amounts to a sequence of  Kronecker foliations $\mathfrak{F}_{n}$ having the same
underlying torus manifold and inessential cycles
$c_{n}$ contained in the leaf through $0$ which eventually are seen to be both horizontal (corresponding to cycles in the homogeneous foliation
$\mathfrak{f}_{n}$) and not, whence the contradiction.

We illustrate how this works in in the case of $n=1$, where Lindemann Weierstra\ss\ reduces to the transcendence of $e$.  We follow the scheme of proof found in \cite{Ba}. 
Suppose that
\[  n_{0}+n_{1}e+\cdots n_{l}e^{l} =0.  \]
Thus, if we let $\bar{e}=(e,\dots ,e^{l})$, in the codimension one Kronecker foliation $\mathcal{F}(\bar{e})$  of $\T^{l+1}$ we have a nontrivial leaf wise cycle given by the equality
\[\bar{n}\cdot\bar{e}=\bar{n}^{\perp}\]
where $\bar{n}=(n_{1},\dots ,n_{k})$ and $\bar{n}^{\perp}=-n_{0}$.

Given $f(x)\in \Z[x]$ non constant, define the factorial
\[ f! := \int_{0}^{\infty} f(x)e^{-t}dt \in\Z .\]
When $f(x)=x^{m}$ we have \[ f!=(m-1)!,\] and in general, $f!$ is a linear combination of ordinary factorials i.e. 
$f!\in \sum \Z n!$.

We will also consider the
 upper and lower incomplete factorials
\[ f!^{-}(s)= \int_{0}^{s} f(x)e^{-t}dt,\quad  f!^{+}(s)= \int_{s}^{\infty} f(x)e^{-t}dt ,\]
so that
\[f! =f!^{-}(s) + f!^{+}(s).\]

When $0<s=m\in\N$ we have
\[e^{m} f!^{+}(m)=f! \in\Z \]
so \[  f!^{+}(m) = f!e^{-m}\]
and therefore $e^{m}f!^{-}(m) = e^{m}f!-f!$ i.e. 
\[ f!^{-}(m) = f!-e^{-m}f! = (1-e^{-m})f!. \]

For each $k$ consider the polynomial
\[ f_{k}(x)= x^{k}(x-1)^{k+1}\cdots (x-l)^{k+1}.\]
Now write
\begin{align*}  f_{k}!\cdot \big(n_{0}+n_{1}e+\cdots n_{l}e^{l} \big) =& 
f_{k}!n_{0} + n_{1}e^{1}f!^{+}_{k}(1)+\cdots +n_{l}e^{l}f!^{+}_{k}(l) \\
 & +  n_{1}e^{1}f!^{-}_{k}(1)+\cdots +n_{l}e^{l}f!^{-}_{k}(l) \\
 = & P_{k}+Q_{k} =0.
\end{align*}
The heart of the argument is that $Q_{k}=-P_{k}$, but $P_{k}/k!\in \Z-0$ yet $Q_{k}/k!$ is positive and eventually $<1$, hence $0$.

From the point of view of Kronecker foliations, $Q_{k}/k!$ corresponds to a closed cycle $c_{k}$ in the Kronecker foliation
$\mathcal{F}(\bar{e}_{k})$ where
\[  \bar{e}_{k}=\big( (f_{k}!^{-}(1)/k! )e,\dots ,(f_{k}!^{-}(l)/k!) e^{l} \big) \]
and $c_{k}$ is defined by
\[\bar{m}\cdot \bar{e}_{k}=\bar{m}^{\perp},\quad \bar{m}=\bar{n},\; \bar{m}^{\perp}=-P_{k}/k!\]
which is eventually for $k$ large, simultaneously,
horizontal (defining a cycle in $\mathfrak{f}( \bar{e}_{k})$ and not horizontal (not defining a cycle in $\mathfrak{f}( \bar{e}_{k}))$ i.e $\bar{m}^{\perp}=0$ and $\not=0$.

\section{Appendix: Diophantine Approximation Groups As Types}\label{types}

The central theme of this article has been the idea that diophantine approximation groups
are the right generalization of ideal for ``real algebraic number theory''.   In this Appendix,
we will discuss how Model Theory reinforces this philosophy 
through the notion of {\it type}.  I would like to express my gratitude
to John Baldwin, for providing guidance and a number of valuable suggestions.  



Basic references for model theory are:  \cite{Ho}, \cite{Ma}, \cite{Po}.  
 Let $\mathcal{L}$ be a language, $M$ an $\mathcal{L}$-structure and $A\subset M$ a subset.   Denote by $\mathcal{L}_{A}$
the language obtained from $\mathcal{L}$  by adding the elements of $A$ as constants, and let ${\rm Th}_{A}(M)$
be the complete $\mathcal{L}_{A}$ theory of $M$ viewed as an $\mathcal{L}_{A}$ structure.  For a variable symbol
$x$ we denote by $\mathcal{L}_{A}(x)$ the language obtained by adding $x$ as a constant.  

By a {\it 1-type} of $M$ (or simply {\it type} if there is no confusion) with parameters
in $A$ is meant a collection of $\mathcal{L}_{A}$-formulas $p$ in a single variable $x$ for which 
\[ {\rm Th}(p) := p\cup {\rm Th}_{A}(M) \] is a
consistent $\mathcal{L}_{A}(x)$-theory.  We say that $p$ is {\it complete} if ${\rm Th}(p)$ is a maximal theory.  For $a\in M$, 
the complete type generated by $x$ is
\[ {\rm tp}(a) = \{ \upphi (x)\; |\;\; M\models \upphi (a)\} .\]
In general, we say that a type $p$
is realized in $M$ if there exists $a\in M$ such that 
$ p\subset {\rm tp}(a) $.
If $p$ is not realized, there is always an elementary extension $N\succ M$ in which $p$ 
is realized.   The Stone space of complete types of $M$ with parameters in $A$ is denoted
$S(M/A)$: with the Stone topology it is a compact and totally disconnected space.  Types in $r$-variables $\bar{x}=(x_{1},\dots ,x_{r})$ are  called {\it $r$-types}.

Types naturally generalize the notion of principal ideal.  If $\mathcal{O}$ is the
ring of integers of an algebraic extension $K/\Q$ and $(\upbeta )$ is the principal ideal generated by $\upbeta\in \mathcal{O}$, then 
the set ${\rm tp}_{\rm lin}(\upbeta /\mathcal{O})\subset {\rm tp}(\upbeta /\mathcal{O})$ of linear equations satisfied by $\upbeta$,
\[  \upphi_{a,b}(x):\;\;\; ax+b=0 ,\quad a,b\in\mathcal{O}, \;\; K\models \upphi_{a,b}(\upbeta ) ,\]
defines an (incomplete) type of $K$ with parameters in $\mathcal{O}$.   We note that \[ (\upbeta )= \{ b\in \mathcal{O} |\; \exists a\in\mathcal{O}\text{ s.t. }\upphi_{a,b}(x)\in {\rm tp}_{\rm lin}(\upbeta /\mathcal{O}) \}.\]

Recall (see \S \ref{nonstd}) the vector space of extended reals $\bbull\R=\bast\R/\bast\R_{\upvarepsilon}$.  We will consider $\bbull\R$ in the language of real vector spaces: 
\[ \mathcal{L}_{\text{\rm  rvs}}=(+,-,0)\cup (f_{r}; \; r\in\R )\]
where $f_{r}$ is the unary function corresponding to scalar multiplication by $r\in \R$.  
Then $\R$ is an  $\mathcal{L}_{{\rm rvs}}$-substructure
of $\bbull\R$ and we have

\begin{prop}\label{elementaryss}  $\R\subset\bbull\R$ is an elementary substructure.
\end{prop}

\begin{proof}  We use the Tarski-Vaught test.  Thus, if $\upphi (x,\bar{s})$ is a formula with parameters
$\bar{s}=(s_{1},\dots , s_{n})$ coming from $\R$ for which there exists $\bbull r\in \bbull\R$ with 
$\bbull\R\models \upphi (\bbull r,\bar{s})$, we must find $r\in\R$ such that
$\R\models \upphi (r,\bar{s})$.  We proceed by induction on complexity of formulas.
If $\upphi (x,\bar{s})$ is atomic it must be either of the form 
\[  \sum f_{r_{i}}(s_{i}) + f_{r'}(\bbull r) =0 \]
or the negation of such a form.
In the positive case we must have $\bbull r\in \R$ already, and in the negated form, one simply
picks a suitable $r$ and $r'$ so that equation is false.   If $\upphi (x,\bar{s})=\exists z\uppsi (x,z,\bar{s})$ 
where $\uppsi$ has no quantifiers, it is easy to see that we may find $x,z$ in $\R$ satisfying $\uppsi$
if there already exist witnesses in $\bbull\R$.
\end{proof}

\begin{note}\label{quantelim}  The proof of Proposition \ref{elementaryss} shows that
the theory of real vector spaces has quantifier elimination.
\end{note}

Since we are essentially interested in formulas which express properties of elements of the scalar field $\R$, we need to introduce variables which range over the elements of $\R$ or more generally
a scale field $R$ which is a model of $\R$.  We can
do this by adding an additional sort for the scalar field and then viewing the scalar multiplication as a {\it single} map.
More precisely, we will consider structures of the type 
\[   (R; V; m)\]
and the many sorted theory ${\sf Vect}_{MS}$ whose axioms are
\begin{itemize}
\item[-] $R=(R,+,\cdot,<, 0,1)$ is a model of the complete theory of the {\it ordered field} $\R$.
\item[-]  $V=(V,+,0)$ is an abelian group.
\item[-]  $m:R\times V\rightarrow V$ satisfies the axioms of scalar multiplication.
\end{itemize}
As usual, in formulas we use $``\cdot "$ for scalar multiplication.
Then $\bbull\R$ and $\R$ are models of ${\sf Vect}_{MS}$, and using a many-sorted version
of the argument appearing in the proof of Proposition \ref{elementaryss}
we may show that $\R\prec \bbull\R$ and that ${\sf Vect}_{MS}$ admits elimination of quantifiers (the order $<$ is added
to achieve quantifier elimination in the scalar sort).

%



We now turn to types in $\bbull\R$.
 Consider $A=\bast\Z\subset\bbull\R$, viewed as a set of parameters in the vector space sort, and let $\uptheta\in\R$.  Consider the complete many sorted type
  ${\rm tp}(\uptheta /\bast\Z)$.  In addition to containing  field formulas satisfied by $\uptheta$ (these having no added field parameters), it also
  contains formulas of the shape $x\cdot \bast n-\bast m=0$ for $\bast m,\bast n\in\bast\Z$.  
  Define
\[  {\rm tp}_{\text{lin}} (\uptheta /\bast\Z ) := \big\{ x\cdot \bast n-\bast n^{\perp} =0 \; |\;\; 
\bast n\in\bast\Z (\uptheta ) \big\} \subset {\rm tp}(\uptheta /\bast\Z) .\]
Then $ {\rm tp}_{\text{lin}} (\uptheta/\bast\Z )$ is an abelian group with respect
to addition of formulas
and \[   {\rm tp}_{\text{lin}} (\uptheta/\bast\Z ) \cong\bast\Z (\uptheta ) .\]
Thus every real diophantine approximation group canonically determines an (incomplete) many sorted type
with parameters in $\bast\Z$.  

In fact, $ {\rm tp}(\uptheta/\bast\Z )$ also contains the group 
$\bast\Z [X](\uptheta )$.
Each polynomial
$g(X) =\bast a_{n}X^{n} +\dots + \bast a_{0}$ in $ \bast \Z [X](\uptheta)$ determines the formula 
\begin{equation}\label{polyform}   \sum_{i=0}^{n} x^{i}\cdot \bast a_{i} =0 . 
\end{equation}
The set of these formulas add like the polynomials that define them
giving a subgroup isomorphic to $ \bast \Z [X](\uptheta)$.


The diophantine approximation group over a ring of integers $\mathcal{O}$ defines
a type contained in ${\rm tp}(\uptheta /\bast\mathcal{O})$.  
For $\Uptheta\in M_{r,s}(\R )$, the diophantine approximation group
$\bast \Z^{s}(\Uptheta )$ defines an $rs$-type
consisting of formulas of the form
\[    (\bast {\boldsymbol n}\cdot \bar{x}_{1}=n_{1}^{\perp})\wedge\dots\wedge
 (\bast {\boldsymbol n}\cdot \bar{x}_{r}=n_{r}^{\perp})
   \]
where $\bar{x}_{i}=x_{i1},\dots ,x_{is}$, $\bast{\boldsymbol n}\in\bast \Z^{s}(\Uptheta )$ and $ \bast {\boldsymbol n}\cdot \bar{x}_{i}$ is short-hand for 
$\sum x_{ij}\cdot\bast n_{j} $.  Thus all of the diophantine approximation groups
defined in this paper define algebraic subtypes of the complete many sorted types defined by their generators.


\end{document}